\pdfoutput=1
\documentclass[reqno]{amsart}

\usepackage[accsupp]{axessibility} 

\usepackage{graphicx,subfigure,color}

\numberwithin{equation}{section}

\usepackage[latin1]{inputenc}
\usepackage[english]{babel}

\usepackage{amsmath,amsthm,amsfonts,latexsym,amssymb}
\usepackage[colorlinks]{hyperref}
\hypersetup{linkcolor=blue,citecolor=blue,filecolor=black,urlcolor=blue}
\usepackage{comment}

\usepackage{color}
\usepackage[dvipsnames]{xcolor}
\definecolor{darkgreen}{rgb}{0,0.7,0.1}

\usepackage{enumerate}

{ \theoremstyle{plain}
\newtheorem{theorem}{Theorem}[section]
\newtheorem{proposition}[theorem]{Proposition}
\newtheorem{lemma}[theorem]{Lemma}
\newtheorem{corollary}[theorem]{Corollary}
  \theoremstyle{remark}
\newtheorem{remark}[theorem]{Remark}
  \theoremstyle{definition}
\newtheorem{definition}[theorem]{Definition}
}

\begin{document}


\title[Multiple BV solutions for a prescribed mean curvature equation]{Multiple bounded variation solutions for a prescribed mean curvature equation with Neumann boundary conditions}

\author[A. Boscaggin]{Alberto Boscaggin}
\address{Alberto Boscaggin\newline\indent 
Dipartimento di Matematica
\newline\indent
Universit\`a di Torino
\newline\indent
via Carlo Alberto 10, 10123 Torino, Italia}
\email{alberto.boscaggin@unito.it}

\author[F. Colasuonno]{Francesca Colasuonno}
\address{Francesca Colasuonno\newline\indent 
Dipartimento di Matematica
\newline\indent
Alma Mater Studiorum Universit\`a di Bologna
\newline\indent
piazza di Porta San Donato 5, 40126 Bologna, Italia}
\email{francesca.colasuonno@unibo.it}

\author[C. De Coster]{Colette De Coster}
\address{Colette De Coster
\newline \indent {Univ. Polytechnique Hauts-de-France, EA 4015 - LAMAV - FR CNRS 2956, 
\newline\indent F-59313 Valenciennes, France}}
\email{colette.decoster@uphf.fr}

\date{\today}

\begin{abstract} 
We prove the existence of multiple positive BV-solutions of the Neumann problem 
$$
\begin{cases}
\displaystyle -\left(\frac{u'}{\sqrt{1+u'^2}}\right)'=a(x)f(u)\quad&\mbox{in }(0,1),\\
u'(0)=u'(1)=0,&
\end{cases}
$$
where $a(x) > 0$ and $f$ belongs to a class of nonlinear functions whose prototype example is given by $f(u) = -\lambda u + u^p$, for $\lambda > 0$ and $p > 1$. In particular, $f(0)=0$ and $f$ has a unique positive zero, denoted by $u_0$. Solutions are distinguished by the number of intersections (in a generalized sense) with the constant solution $u = u_0$. We further prove that the solutions found have continuous energy and we also give sufficient conditions on the nonlinearity to get classical solutions. The analysis is performed using an approximation of the mean curvature operator and the shooting method. 
\end{abstract}

\maketitle

\section{Introduction}

In the last decades, a great deal of research has been devoted to the study of nonlinear boundary value problems associated with the mean curvature equation
\begin{equation}\label{eq-intro}
\textnormal{div}\left( \frac{\nabla u }{\sqrt{1 + \vert \nabla u \vert^2}}\right) + g(x,u) = 0, \qquad x \in \Omega \subset \mathbb{R}^N,
\end{equation}
both in the ODE case ($N = 1$) and in the PDE one ($N \geq 2$); see, among many others, \cite{BHOO,CG,gerhardt,giusti,le,marzocchi,OO,pomponio,serrin} and the references therein. Besides this well known interpretation from Differential Geometry, this equation also appears in several contexts from Mathematical Physics, such as reaction-diffusion processes with saturation at high regimes \cite{BG,KR}, capillarity phenomena for incompressible fluids \cite{finn,huisken}, modeling of the human cornea \cite{CDOOS,CDO,CDOO,OP}. From the genuinely mathematical point of view, the investigation of equation \eqref{eq-intro} leads to a variety of challenging technical issue, since, due to the strongly nonlinear character of the differential operator, it becomes necessary to take into account weaker notions of solutions, possibly exhibiting jump discontinuities. 

Along this line of research, in this paper we deal with the following one-dimensional Neumann problem 
\begin{equation}\label{eq:Pmain}
\begin{cases}
\displaystyle -\left(\frac{u'}{\sqrt{1+u'^2}}\right)'=a(x)f(u)\quad&\mbox{in }(0,1),\\
u>0&\mbox{in }(0,1),\\
u'(0)=u'(1)=0,&
\end{cases}
\end{equation}
where $a\in C^1([0,1])$, $a>0$ in $[0,1]$, and $f\in C^1([0,\infty))$ is a nonlinear term whose prototype example is given by
\begin{equation}\label{eq:proto}
f(s) = -\lambda s + s^p, \qquad \lambda > 0, \;p > 1.
\end{equation}
In particular, problem \eqref{eq:Pmain} has a unique constant solution $u \equiv u_0$  and we are interested in studying existence, multiplicity and some qualitative properties of non-constant solutions of \eqref{eq:Pmain} that oscillate around $u_0$.

The choice for this nonlinear term is partially inspired by some recent results, dealing with the radial Neumann problem (in an annulus or in a ball) associated with the semilinear equation 
$$
-\Delta u = f(u),
$$
see \cite{BonheureGrossiNorisTerracini2015,bonheure2016multiple,BNW,ABF2,ABF1,ma2016bonheure}, 
and with the Minkowski-curvature equation 
$$
\textnormal{div}\left( \frac{\nabla u }{\sqrt{1 - \vert \nabla u \vert^2}}\right) +f(u) = 0,
$$
see \cite{BCN4,BCN3}. In the above papers, it is shown that, for a large class of nonlinear terms $f$ including \eqref{eq:proto}, non-constant positive radial solutions oscillating around $u_0$ can be provided: more precisely, radial solutions $u$ having exactly $k$ intersections with $u_0$ exist if $f'(u_0)$ is greater than the $k$-th non-zero eigenvalue of the radial Neumann problem for $-\Delta u = \lambda u$.
On growing of the value $f'(u_0)$, a high multiplicity of solutions thus appear, confirming a conjecture first given in \cite{BNW}; in all these papers, solutions are meant in the classical sense. Notice also that
the Minkowski-curvature operator behaves as $\Delta u$ when $\nabla u$ is small, this being the reason why in both cases the condition required on $f'(u_0)$ to guarantee the existence of non-constant solutions is related to the eigenvalues of $-\Delta u=\lambda u$, cf. \cite[Theorem 1.1]{BCN3}.

The aim of this paper is to provide a similar solvability pattern for the boundary value problem \eqref{eq:Pmain}.
As in the recent papers \cite{LGO,LGO1,LGO2,LGOR1,LGOR2}, and since we will take advantage of some regularity results proved therein, we choose here to work in a purely one-dimensional setting; however, to avoid trivialities, we assume that a non-constant weight $a(x)$ can appear in front of the nonlinear term $f$. Notice that, since the mean curvature operator linearizes as $u''$ for $u'$ small, here the eigenvalues of the weighted problem $-u'' = \lambda a(x) u$ are expected to play a role. 

As already anticipated, the main difficulties in considering problem \eqref{eq:Pmain} are due to the possible lack of regularity of the solutions. 
We work indeed with Bounded Variation (BV, for short) solutions to \eqref{eq:Pmain}, a variational notion of solutions basically going back to the works of   A. Lichnewsky and R. Temam \cite{Te, EkTe, Li74a, Li74b, Li78, LiTe} and  E. Giusti and M. Miranda in \cite{Mi77, GiuIM, GiuPJM, Mi} and now commonly used in this context; we recall the precise definition at the beginning of Section \ref{sec:2} for the reader's convenience.  
The crucial point, here, is to define, for a (possibly discontinuous) BV-solution, a suitable notion of intersection with the constant $u_0$, so as to provide their multiplicity. To the best of our knowledge, a similar issue has never been faced in the context of mean curvature equations, for which only few high multiplicity results are known  \cite{Obe,HO,OO09,OO10,OO11,COZ}.

We now state our result precisely. First, we introduce the following structural assumptions on the nonlinear term: for some positive constant $u_0>0$, it holds
\begin{itemize}
\item[$(f_\mathrm{eq})$] $f(0)=f(u_0)=0$;
\item[$(f_\mathrm{sgn})$] $f(s)<0$ if $s\in(0,u_0)$, $f(s)>0$ if $s\in(u_0,\infty)$.
\end{itemize}
Moreover, we will also suppose that $f$ satisfies one of the following two conditions: 
\begin{itemize}
\item[$(f_\mathrm{ap})$] there exists $\bar u>u_0$ such that 
$$
\int_{u_0}^{\bar u}f(s)ds = C_a \int_{u_0}^{0}f(s)ds,
$$
where $C_a:=\max\Big\{\frac{a(1)}{a(0)}\exp \big(\int_0^1 \frac{{a'}^-(x)}{a(x)}dx \big), 
\frac{a(0)}{a(1)}\exp \big(\int_0^1 \frac{{a'}^+(x)}{a(x)}dx\big)\Big\}$;
\item[$(f_\mathrm{ap})'$] $\|a\|_{L^1(0,1)} \max f^- < 1$.
\end{itemize}
Clearly enough, the above conditions are completely unrelated. In particular, $(f_\mathrm{ap})$ is a condition on the behavior of $f$ at infinity;
it is surely satisfied when $\int_{u_0}^{+\infty}f(s)ds = +\infty$ so that, in particular, the model nonlinearity \eqref{eq:proto} fulfills it for every $\lambda >0$ and $p > 1$. Incidentally, notice also that, in case $a$ is monotone, the constant $C_a$ reduces to  $\max_{[0,1]} a/\min_{[0,1]}a$.
On the other hand, $(f_\mathrm{ap})'$ concerns only the behavior of $f$ in $[0,u_0]$, cf. $(f_\mathrm{sgn})$. Assumption $(f_\mathrm{ap})'$ is inspired from some arguments in \cite{OO2}, see also \cite{Rthesis}. It is satisfied by every function $f$ satisfying $(f_\mathrm{sgn})$, at the cost of asking that $\|a\|_{L^(0,1)}$ is small enough. Moreover, in view of \cite[Lemma 3.1]{LGO}, $(f_\mathrm{ap})'$ seems to be quite natural when looking for solutions with $u(0)<u_0$.

Second, 
we define the energy $\mathcal{E}$ for a solution of \eqref{eq:Pmain} by formally letting
\begin{equation}\label{eq:energyPmain}
\mathcal E(x):=1-\frac{1}{\sqrt{1+(u'(x))^2}}+a(x)F(u(x)),
\end{equation}
where $F(u) := \int_{u_0}^u f(s)ds$. 
Of course, in principle this definition is meaningless for a BV-function; however, it will be clear from the statement of the result that $\mathcal{E}$ is well-defined for every value of $x$ in the interval $[0,1]$ up to a finite number of points where $u$ is discontinuous 
(if $u$ is continuous but not differentiable at some point $\bar{x}$, it must be $\vert u'(\bar{x}) \vert = +\infty$ and the definition of $\mathcal{E}$ has to be intended in the limit sense, i.e., $\mathcal E(\bar{x})=1+a(\bar{x})F(u(\bar{x}))$. 

Finally, we introduce, for every $k\in\mathbb N$, $\lambda_{k}$ as the $k$-th eigenvalue of $-u'' = \lambda a(x) u$ in $(0,1)$ with Neumann boundary conditions;  namely, $0=\lambda_1<\lambda_2<\lambda_3<\dots$ and $\lambda_k\to \infty$ as $k\to\infty$. Moreover, we denote by  $u'(x^-)$ and $u'(x^+)$  the left and the right Dini derivatives of $u$ at $x \in (0,1)$, respectively. 

We can then state our result as follows.

\begin{theorem}\label{thm:main}
Let $a\in C^1([0,1])$ be such that $a(x)>0$ for every $x \in [0,1]$. 
Let $f\in C^1([0,\infty))$ satisfy $(f_\mathrm{eq})$, $(f_\mathrm{sgn})$ and either $(f_\mathrm{ap})$ or $(f_\mathrm{ap})'$.
Moreover, let us suppose that, for some $k\in \mathbb N$, 
$$
f'(u_0) >\lambda_{k+1}.
$$
Then, there exist at least $2k$ distinct non-constant BV-solutions $u_1,\,\dots, u_{2k}$ of \eqref{eq:Pmain}. Furthermore, 
\begin{enumerate}[(I)]
\item
for every $j=1,\dots,k$, there exist exactly $j+2$ points $0=x_{j,0}< x_{j,1} < \dots < x_{j,j}<x_{j,j+1}=1$, such that 

\begin{enumerate}[(a)]
\item $u_j\in C^2([0,x_{j,1}))\cap C^2((x_{j,j},1])$,  $u_j\in C^2((x_{j,i},x_{j,i+1}))$ for $i=1,\dots,j-1$ and $u_j'(0)=0=u_j'(1)$;
\item $(-1)^{i+1}(u_j(x)-u_0)>0$ for every $x\in (x_{j,i},x_{j,i+1})$ for $i=0,\dots,j$;
\item for every $i=1,\dots,j$ one of the following two statements holds true: 
\begin{enumerate}[(i)]
\item $u_j(x_{j,i})=u_0$ and $u_j\in C^2((x_{j,i-1},x_{j,i+1}))$;
\item $u_j(x_{j,i}^-)\le u_0 \le u_j(x_{j,i}^+)$ and $u_j'(x_{j,i}^-)=+\infty=u_j'(x_{j,i}^+)$ if $i$ is odd, 
$u_j(x_{j,i}^+)\le u_0 \le u_j(x_{j,i}^-)$ and $u_j'(x_{j,i}^-)=-\infty=u_j'(x_{j,i}^+)$ if $i$ is even,
\end{enumerate}
\end{enumerate}
\item
for every $\ell=1,\dots,k$, there exist exactly $\ell+2$ points $0=x_{\ell,0}< x_{\ell,1} < \dots < x_{\ell,\ell}<x_{\ell,\ell+1}=1$, such that 
\begin{enumerate}[(a)]
\item $u_{2k+1-\ell}\in C^2([0,x_{\ell,1}))\cap C^2((x_{\ell,\ell},1])$, $u_{2k+1-\ell}\in C^2((x_{\ell,i},x_{\ell,i+1}))$ for $i=1,\dots,\ell-1$ and $u_\ell'(0)=0=u_\ell'(1)$;
\item $(-1)^{i+1}(u_{2k+1-\ell}(x)-u_0)<0$ for every $x\in (x_{\ell,i},x_{\ell,i+1})$ for $i=0,\dots,\ell$;
\item for every $i=1,\dots,\ell$ one of the following two statements holds true: 
\begin{enumerate}[(i)]
\item $u_{2k+1-\ell}(x_{\ell,i})=u_0$ and $u_{2k+1-\ell}\in C^2((x_{\ell,i-1},x_{\ell,i+1}))$;
\item $u_{2k+1-\ell}(x_{\ell,i}^-)\le u_0 \le u_{2k+1-\ell}(x_{\ell,i}^+)$ and $u_{2k+1-\ell}'(x_{\ell,i}^-)=+\infty=u_{2k+1-\ell}'(x_{\ell,i}^+)$ if $i$ is even, 
\\
$u_{2k+1-\ell}(x_{\ell,i}^+)\le u_0 \le u_{2k+1-\ell}(x_{\ell,i}^-)$ and $u_{2k+1-\ell}'(x_{\ell,i}^-)=-\infty=u_{2k+1-\ell}'(x_{\ell,i}^+)$ if $i$ is odd,
\end{enumerate}
\end{enumerate}
\item
for every $j=1,\dots,2k$
the energy $\mathcal E_j$ corresponding to the solution $u_j$ can be extended by continuity to $[0,1]$.
\end{enumerate}
\end{theorem}

We observe that in part (I) of this theorem, we describe the oscillating solutions having $u(0)<u_0$, while part (II) deals with solutions having $u(0)>u_0$. Moreover, in both cases (I)-(c)-(i) and (I)-(c)-(ii), for every $i=1,\dots,j$ the function $u_j(\cdot)-u_0$ changes sign exactly once in the interval $(x_{j,i-1},x_{j,i+1})$, at point $x_{j,i}$.  When the case (c)-(ii) occurs, we will informally refer to $x_{j,i}$ as a {\it generalized intersection point} of $u$ with $u_0$. 
In this way, we can summarize parts (I)-(c)-(i) and (I)-(c)-(ii) of Theorem \ref{thm:main} stating that $u_j$ has exactly $j$ generalized intersections with $u_0$, that occur at points $x_{j,1},\dots, x_{j,j}$. A similar remark can be done also for parts (II)-(c)-(i) and (II)-(c)-(ii) of the statement. We also emphasize Part (III) of Theorem~\ref{thm:main}, ensuring that the solutions found have continuous energy: this is not always the case for a general BV-solution and it will be obtained as a consequence of our method of proof. 

The necessity of taking into account possibly discontinuous solutions with generalized intersections with $u_0$ is well-recognized even in the autonomous case, $a(x) \equiv a$. In such a case the equation in \eqref{eq:Pmain} can be equivalently written as the planar Hamiltonian system
\begin{equation}\label{ham-sys}
u' = \frac{v}{\sqrt{1-v^2}}, \qquad v' = - a f(u),
\end{equation}
and solutions lie on level sets of the energy
$$
H(u,v) = 1 - \sqrt{1-v^2} + a F(u),
$$
where, again, $F(u) = \int_{u_0}^u f(s)ds$; incidentally, notice that this agrees with our previous definition of the energy $\mathcal{E}$, 
that is, $H(u(x),v(x)) = \mathcal{E}(x)$ if $(u,v)$ is a solution of \eqref{ham-sys}. An elementary phase-plane analysis shows that the point $(u_0,0)$ is a global minimum of $H$ and, hence, a local center for the system, being surrounded by classical closed orbits with periods tending to $2\pi/\sqrt{f'(u_0)}$, as the solution shrinks to $u_0$. 
These orbits give rise to natural candidate solutions to \eqref{eq:Pmain}: precisely, if $(u,v)$ is a closed orbit around $(u_0,0)$ satisfying 
$(u(0),v(0)) \in \mathbb{R}^+ \times \{0\}$, then $u$ satisfies the Neumann boundary conditions if and only if the half-period is of the type $1/\ell$ for some $\ell \in \mathbb{N}$; in such a case, the solution $u$ is classical and the number of intersections with $u_0$ is exactly $\ell$. 
Discontinuous BV-solutions, on the other hand, have to be found at energy levels giving rise to disconnected level sets: due to the specific form of $H$, it is easy to check that, for an orbit $(u,v)$ intersecting the positive $u$-semiaxis at a point $(\tilde{u},0)$, this happens if and only if $a F(\tilde{u}) \geq 1$. In this regard, it is worth recalling that the BV-solutions provided by Theorem \ref{thm:main}, even when discontinuous, still have a continuous energy. Therefore, in the autonomous case BV-solutions are obtained by switching to different connected components of the same energy level set. See also Figure \ref{fig}.  

\begin{figure}\label{fig}
\centering
\includegraphics[scale=.5]{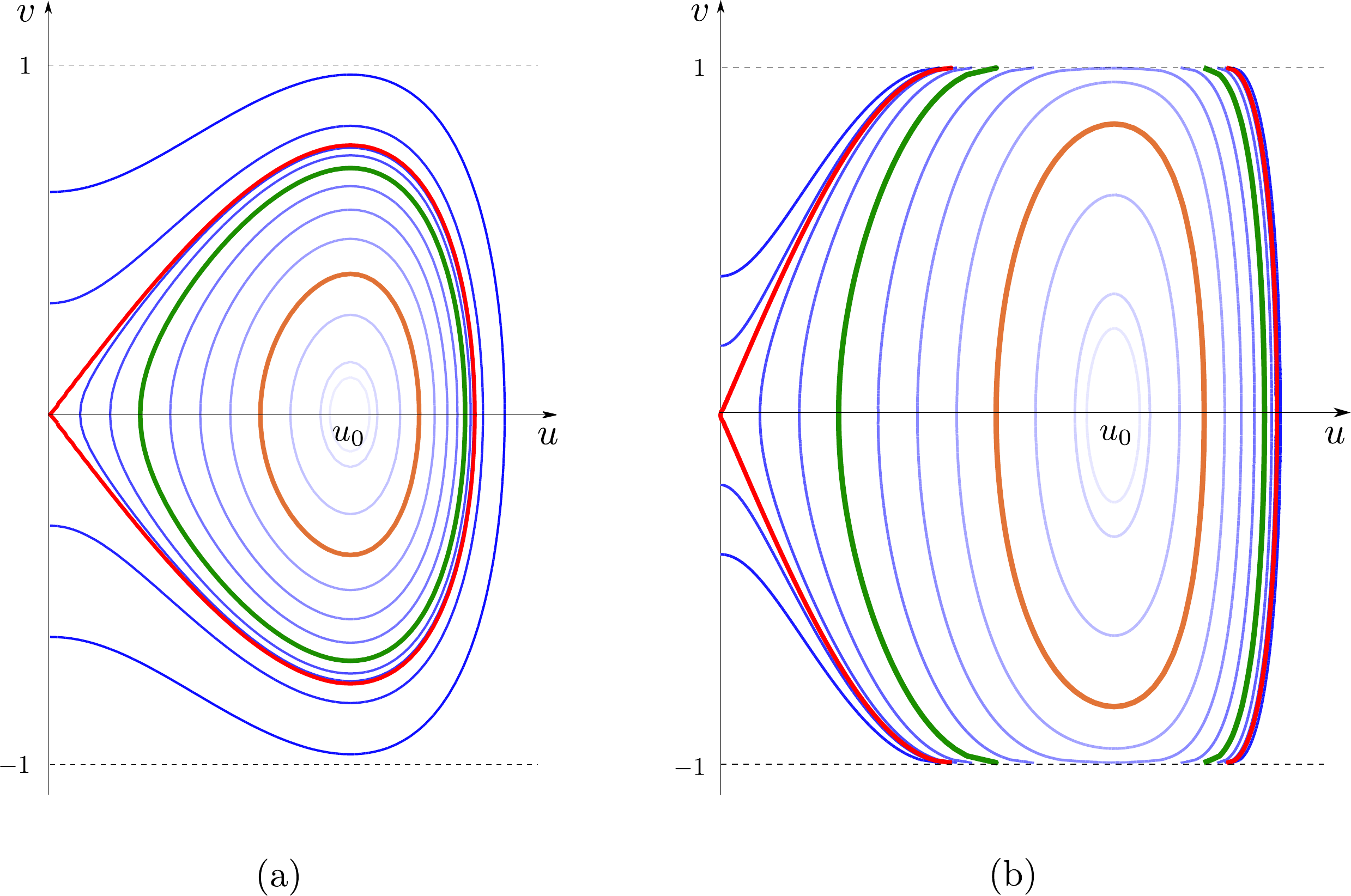}
\caption{Phase-portraits of the planar Hamiltonian system \eqref{ham-sys}, for different suitable choices of $a$ and $f$ (incidentally, let us recall that $\vert v \vert = 1$ if and only if $\vert u' \vert = +\infty$). 
On the left, all Neumann solutions to the equation are classical ones, corresponding to closed orbits around the equilibrium point $(u_0,0)$ (two of them are painted in green and orange color, respectively). On the right, on the contrary, due to the breakdown of some energy levels as soon as $\vert v \vert$ reaches the value $1$, both classical (the orange one, corresponding to a closed orbit) and discontinuous BV-solutions (the green one, corresponding to different connected components of the same energy level) appear. It is worth noticing that the existence of a classical homoclinic orbit to the saddle equilibrium point $(0,0)$ (painted in red in the figure on the left) is a sufficient condition for all the Neumann solutions to be classical: based on this observation, we will establish a similar criterion for the non-autonomous case in Section \ref{sec:6}.}
\end{figure} 

The proof of this theorem is quite long and we prefer to describe here the strategy and the main ideas behind it. In order to do this, it is useful to write the differential operator driving the equation in \eqref{eq:Pmain} as $(\varphi(u'))'$ with $\varphi(s):=s/\sqrt{1+s^2}$: it thus becomes apparent that the possible lack of regularity of the solutions is a consequence of the boundedness of $\varphi(\mathbb R)$.

To overcome this difficulty, we approximate $\varphi$ with the sequence $(\varphi_n)_n$ of $C^1$-functions that coincide with $\varphi$ in $[-n,n]$ and are affine in $\mathbb R\setminus[-n,n]$. We first study, for every $n\in\mathbb N$, the approximated problem \eqref{eq:Pn} governed by the operator $-(\varphi_n(u'))'$ (cf. \cite{BHOO,LGOR1} for a similar strategy) and we prove that, if $f'(u_0) >\lambda_{k+1}$, each approximated problem \eqref{eq:Pn} has $2k$ classical solutions. 
This multiplicity result is obtained via shooting method. In particular, we consider the Cauchy problem $u(0)=d$, $u'(0)=0$ associated to the equation $-(\varphi_n(u'))'=a(x)f(u)$ of the approximated problem \eqref{eq:Pn} and look for values of $d\in\mathbb R^+$ for which the solution $u_d$ of the Cauchy problem satisfies $u_d'(1)=0$, thus solving also \eqref{eq:Pn}.
To this aim, we count the number of half-turns performed by the solution in the phase plane around the equilibrium $(u_0,0)$, as done in \cite{ABF1,ABF2} for a $p$-Laplacian problem. We observe that neither $(f_\mathrm{ap})$ nor $(f_\mathrm{ap})'$ is required for the multiplicity result of the approximated problem.

We then prove that each sequence of solutions of the approximated problems (with a fixed number of intersections with $u_0$) converges in some sense to a BV-solution of the original problem \eqref{eq:Pmain}. To this aim, under either of the assumptions $(f_\mathrm{ap})$ and $(f_\mathrm{ap})'$, we prove that these sequences are bounded in $W^{1,1}(0,1)$, cf. Lemmas \ref{lem:bdd-Linfty} and \ref{lem:uprime-bdd-L1'}, and so it weakly-* converges to a BV-function, up to a subsequence. Once we have a limit function for every sequence, we prove in Proposition \ref{prop:uBvsol} that those functions actually are BV-solutions of \eqref{eq:Pmain}. 

The most delicate point is now to distinguish the $2k$ BV-solutions that we obtained with this approximation procedure. 
We manage to do that, by proving that each limit function inherits the oscillatory behavior of the approximating solutions. This is based on the Propositions \ref{prop:vn-conv-unif} and \ref{prop:un-conv-unif}, that ensure that the convergence is actually much stronger (viz. uniform) away from the intersection points. As a consequence, we get that the BV-solutions of \eqref{eq:Pmain} are allowed to jump only at generalized intersection points. 

We finally prove that the BV-solutions inherit also another important property from the approximating solutions, that is the continuity of the energy. This is mainly based on the preliminary Lemma~\ref{lem:u'ntou'} which ensures that, away from the generalized intersection points, the convergence is even $C^1$. 

The paper is organized as follows. In Section \ref{sec:2}, we give some preliminary known results and useful consequences about BV-solutions and the associated linear eigenvalue problem. In Section \ref{sec:3}, we study the approximated problems via shooting method and prove some properties of the approximating solutions. Section \ref{sec:4} is devoted to the proof of Parts (I) and (II) of the main result of the paper, Theorem \ref{thm:main}, while in Section \ref{sec:5} we establish Part (III) of the same result. Finally, in Section~\ref{sec:6}, we give some sufficient conditions on the nonlinearity $f$ to get classical solutions of \eqref{eq:Pmain}.

\section{Preliminaries}\label{sec:2}

In this section, we state some known results that will be useful in the subsequent sections. 
We start by clarifying the notion of solution used throughout the paper. 

\begin{definition}
We say that $u\in BV(0,1)$ is a \emph{bounded variation solution} (BV-solution) of problem \eqref{eq:Pmain} if for every $v\in BV(0,1)$
\begin{equation}\label{eq:defBV}
\int_0^1\sqrt{1+|Dv|^2}\ge \int_0^1\sqrt{1+|Du|^2}+\int_0^1 a(x)f(u)(v-u)dx,
\end{equation}
where 
$$
\int_0^1\sqrt{1+|Dv|^2}:=\sup\left\{\int_0^1(v w_1' + w_2)dx\,:\, w_1,\, w_2\in C^1_c(0,1), \|w_1^2+w_2^2\|_\infty \le 1 \right\}.
$$
Equivalently, $u\in BV(0,1)$ is a BV-solution of \eqref{eq:Pmain} if it is a global minimizer of the functional $I_u:BV(0,1)\to \mathbb R$ defined as
$$
I_u(v):=\int_0^1\sqrt{1+|Dv|^2}-\int_0^1 a(x)f(u)vdx\quad\mbox{for every }v\in BV(0,1).
$$
Incidentally, let us recall that, from \cite[Proposition 2.36]{BGH}, $BV(0,1)$ embeds into $L^{\infty}(0,1)$ so that the integrals are well-defined.
\end{definition}

\begin{remark}\label{rmk:weak-BV} Observe that by \cite[Lemma 2.1]{LGO}, $u\in W^{1,1}(0,1)$ is a weak solution of  \eqref{eq:Pmain} (in the usual sense), if $u$ satisfies   \eqref{eq:defBV} for all $v\in W^{1,1}(0,1)$. Moreover by \cite[Lemma 2.3]{LGO},  a weak solution $u\in W^{1,1}(0,1)$  of  \eqref{eq:Pmain} is also a BV-solution.
\end{remark}

\subsection{Some approximation lemmas}
\begin{lemma}\label{lem:W11toBV}
For every $u\in BV(0,1)$ there exists a sequence $(u_n)\subset W^{1,1}(0,1)$ such that 
\begin{itemize}
\item[(i)] $u_n\to u$ in $L^1(0,1)$;
\item[(ii)] $\int_0^1 |u_n'|dx\to\int_0^1 |Du|$;
\item[(iii)] $\int_0^1\sqrt{1+u_n'^2}dx\to\int_0^1\sqrt{1+|Du|^2}$.
\end{itemize} 
\end{lemma}

\begin{proof} See \cite[Fact 3.1 and 3.3]{A}.
\end{proof}

\begin{proposition}\label{prop:ineqW11ineqBV}
If $u\in BV(0,1)$ satisfies the inequality in \eqref{eq:defBV} for every $v\in C^\infty([0,1])$, then it is a $BV$-solution of \eqref{eq:Pmain}.
\end{proposition}

\begin{proof} 
By  \cite[Lemma 2.2]{LGO}, we know that $u$ is a BV solution if and only if   \eqref{eq:defBV} is satisfied for all $v\in W^{1,1}(0,1)$. Then we conclude by density of $C^\infty([0,1])$ in $W^{1,1}(0,1)$ (see
\cite[Theorem 8.7]{B}) together with the continuous embedding of $BV(0,1)$ in $L^{\infty}(0,1)$ (see \cite[Proposition 2.36]{BGH}).
\end{proof}

\subsection{Regularity and qualitative properties of BV-solutions to \eqref{eq:Pmain}}

\begin{proposition}\label{prop:Bvsol-description}
Let $u$ be a BV-solution of \eqref{eq:Pmain}, then the following statements hold true. 
\begin{itemize}
\item[(i)] $\displaystyle{\int_0^1 a(x)f(u(x))dx=0}$.\smallskip
\item[(ii)] Let $(\alpha,\beta)\subset(0,1)$ be an interval such that $u(x)\in[u_0,\infty)$ for a.e. $x\in(\alpha,\beta)$ (resp., $u(x)\in[0,u_0]$ for a.e. $x\in(\alpha,\beta)$). Then, $u$ is concave (resp., convex) in $(\alpha,\beta)$, and its restriction to $(\alpha,\beta)$ is of class $C^2((\alpha,\beta))\cap W^{1,1}(\alpha,\beta)$ and satisfies $-\left(\frac{u'}{\sqrt{1+u'^2}}\right)'=a(x)f(u)$ for every $x \in (\alpha,\beta)$.

\noindent Furthermore, if $\alpha=0$, $u\in C^2([0,\beta))$ and $u'(0)=0$, and similarly, if $\beta=0$, $u\in C^2((\alpha,1])$ and $u'(1)=0$. \smallskip
\item[(iii)] Let $(\alpha,\beta)$ and $(\beta,\gamma)$ be any pair of adjacent subintervals of $(0,1)$ such that $u(x)\in[u_0,\infty)$ for a.e. $x\in(\alpha,\beta)$ and  $u(x)\in[0,u_0]$ for a.e. $x\in(\beta,\gamma)$ (resp., $u(x)\in[0,u_0]$ for a.e. $x\in(\alpha,\beta)$ and  $u(x)\in[u_0,\infty)$ for a.e. $x\in(\beta,\gamma)$). Then, either $u\in C^2((\alpha,\gamma))$, or 
$u(\beta^-)\ge u(\beta^+)$ and $u'(\beta^-)=-\infty=u'(\beta^+)$ (resp., $u(\beta^-)\le u(\beta^+)$ and $u'(\beta^-)=+\infty=u'(\beta^+)$).  
\end{itemize}
\end{proposition}

\begin{proof}
(i) If we take $v:=u-1$ as test function in \eqref{eq:defBV}, we get 
$$
\int_0^1 a(x)f(u(x))dx\ge 0,
$$
while, if we take $v:=u+1$ as test function in \eqref{eq:defBV}, we get 
$$
\int_0^1 a(x)f(u(x))dx\le 0.
$$
Therefore, the only possibility is that (i) holds. In view of assumption $(f_\mathrm{sgn})$, and taking into account the fact that $a$ is continuous and positive, parts (ii) and (iii) are immediate consequences of \cite[Proposition~3.6]{LGO}.
\end{proof}

\begin{remark}
By Proposition \ref{prop:Bvsol-description}-(i), in view of assumption $(f_\mathrm{sgn})$, we get that if $u$ is a BV-solution of \eqref{eq:Pmain}, then neither of the following two conditions can be verified
\begin{itemize}
\item $u(x)\in (0,u_0)$ for a.e. $x\in(0,1)$;
\item $u(x)\in(u_0,\infty)$ for a.e. $x\in(0,1)$.
\end{itemize}
\end{remark}

As a consequence of Proposition \ref{prop:Bvsol-description}-(ii), we can obtain the following result, which basically ensures that a BV-solution 
$u$ of \eqref{eq:Pmain}, if not identically equal to $u_0$, can assume the value $u_0$ only at points where $u(\cdot)-u_0$ changes sign. 

\begin{corollary}\label{cor:inter}
Let $u$ be a BV-solution of \eqref{eq:Pmain}. If there exist $j+2$ points $0=x_{0}< x_{1} < \dots < x_{j}<x_{j+1}=1$ such that 
$(-1)^{i+1}(u(x)-u_0) \geq 0$ for every $x\in (x_{i},x_{i+1})$ for $i=0,\dots,j$ (resp., $(-1)^{i}(u(x)-u_0) \geq 0$ for every $x\in (x_{i},x_{i+1})$ for $i=0,\dots,j$), then
either $u \equiv u_0$ on $(0,1)$ or $u(x) \neq u_0$ for every $x \in (0,1) \setminus \{ x_1,\dots,x_j \}$.
\end{corollary}

For the proof, we adopt the following elementary version of the strong maximum principle: if $u: (a,b) \to \mathbb{R}$ is a convex (resp., concave) function of class $C^1$ and if $x_0 \in (a,b)$ is a maximum (resp., minimum) point of $u$, then $u$ is constant on $(a,b)$.

\begin{proof}
Assume that there exists $x^* \in (x_{i},x_{i+1})$ for some $i \in \{0,\dots,j\}$ such that $u(x^*) = u_0$;
moreover, to fix the ideas suppose that  $u(x) \leq u_0$ for every $x \in (x_i, x_{i+1})$ so that, by part (ii) of Proposition \ref{prop:Bvsol-description}, $u$ is convex on such an interval. Therefore, the strong maximum principle yields $u(x) = u_0$ for every $x \in (x_i, x_{i+1})$. Hence,  $u(x) \geq u_0$ for every $x \in (x_{i-1}, x_{i+2})$ (if $i = 0$ or $i = j$, we work on the intervals 
$(x_i,x_{i+2})$ or $(x_{i-1},x_{i+1})$). By part (ii) of Proposition \ref{prop:Bvsol-description}, $u$ is concave on $(x_{i-1}, x_{i+2})$ and the strong maximum principle can be applied again to obtain $u \equiv u_0$ on $(x_{i-1}, x_{i+2})$.
By repeating the argument, $u \equiv u_0$ on $(0,1)$.
\end{proof}

\subsection{The associated eigenvalue problem}
We consider the eigenvalue problem associated to \eqref{eq:Pmain}, namely
\begin{equation}\label{eq:eigenP}
\begin{cases}
-u''=\lambda a(x)u\quad&\mbox{in }(0,1),\\
u'(0)=u'(1)=0,
\end{cases}
\end{equation}
where $a=a(x)$ is the same positive weight appearing in \eqref{eq:Pmain}. For \eqref{eq:eigenP}, the following classical result holds, cf. for instance \cite{H}.
\begin{theorem}\label{thm:spectrum}
The eigenvalues of \eqref{eq:eigenP} form  a divergent, increasing sequence   $0=\lambda_1<\lambda_2<\dots\lambda_k<\dots$, $\lim_{k\to\infty}\lambda_k = \infty$. Moreover, every eigenvalue is simple. 
The eigenfunction that corresponds to the $k$-th eigenvalue $\lambda_k$
has exactly $k-1$ simple zeros in $(0,1)$.
\end{theorem}
For future use, we introduce the clockwise polar coordinates
\begin{equation*}
\begin{cases}
u(x)=\varrho(x)\cos\vartheta(x)&\\
u'(x)=-\varrho(x)\sin\vartheta(x)&
\end{cases}
\end{equation*}
and we write the equation satisfied by the angular variable $\vartheta$ of a solution $u$ of \eqref{eq:eigenP} 
\begin{equation}\label{eq:theta'-eigen}
\vartheta_\lambda'(x)=\sin^2\vartheta(x)+\lambda a(x)\cos^2\vartheta(x)\quad\mbox{for }x\in(0,1).
\end{equation}
Notice that $\vartheta_\lambda'(x)>0$ for every $x\in (0,1)$, and so $\vartheta_\lambda=\vartheta_\lambda(x)$ is strictly increasing. Hence, the oscillatory behavior of the eigenfunctions of \eqref{eq:eigenP} described in Theorem~\ref{thm:spectrum} can be expressed in the following way: if $\lambda=\lambda_{k+1}$, the angular variable $\vartheta_{\lambda_{k+1}}$ that corresponds to the $(k+1)$-th eigenfunction satisfies the identity 
\begin{equation}\label{eq:rotation-eigen}
\vartheta_{\lambda_{k+1}}(1)-\vartheta_{\lambda_{k+1}}(0)=k\pi.
\end{equation}
We further observe that also the map $\lambda\mapsto\vartheta_\lambda$ is strictly increasing, in the sense that if $\lambda< \mu$, then the following implication holds
\begin{equation}\label{eq:theta-lambda-increas}
\vartheta_\lambda(0)\le \vartheta_\mu(0)\quad\Rightarrow\quad \vartheta_\lambda(x)< \vartheta_\mu(x) \mbox{ for every }x\in (0,1),
\end{equation}
cf. for instance \cite[Theorem 4]{RW99}. 
By convention, we choose an eigenfunction $u$ satisfying $u(0) > 0$, and so we
couple \eqref{eq:theta'-eigen} with the initial condition $\vartheta_\lambda(0)=0$ for every $\lambda$. 

\section{The approximated problem}\label{sec:3}
Let $\varphi(s):=\frac{s}{\sqrt{1+s^2}}$ for every $s\in\mathbb R$. For every $n\in\mathbb N$, we introduce $\varphi_n:\mathbb R\to \mathbb R$ as the $C^1$-function such that 
$$
\varphi_n(s)=
\begin{cases}
\varphi(s)\quad&\mbox{if }|s|\le n,\\
\mbox{affine} &\mbox{if }|s|> n.
\end{cases}
$$
\begin{proposition}\label{prop:phin-Phin}
Let $\Phi$ and $\Phi_n$ be the primitives of $\varphi$ and $\varphi_n$ that vanish in zero, namely 
$$
\Phi(s):=\displaystyle{\int_0^s\varphi(\xi)d\xi=\sqrt{1+s^2}-1} \quad\mbox{and}\quad \Phi_n(s):=\displaystyle{\int_0^s\varphi_n(\xi)d\xi}.
$$
For every $n\in\mathbb N$, $\varphi$, $\varphi_n$, $\Phi$, and $\Phi_n$ enjoy the following properties:
\begin{itemize}
\item[(a)] $\varphi_n(s)s\ge \varphi_{n+1}(s)s\ge \varphi(s)s\geq 0$ for every $s\in\mathbb R$;
\item[(b)] $\varphi_n$ is increasing and $\Phi_n$ is convex;
\item[(c)] $\Phi_n(s)\ge \Phi(s)$ for every $s\in\mathbb R$;
\item[(d)] $\Phi_n(s)\le \varphi_n(s)s$ for every $s\in\mathbb R$ and $\lim_{s\to\pm\infty}(\varphi_n(s)s-\Phi_n(s))=+\infty$;
\item[(e)] $s\varphi_n^{-1}(s)\ge s^2$ for every $s\in\mathbb R$.
\end{itemize}
\end{proposition}

\begin{proof}
By definition, $\varphi$ and $\varphi_n$ are odd functions, and so we can restrict the proof of the properties to $s\ge 0$. To prove (a), we observe that being $\varphi$ concave, $\varphi_n\ge \varphi$ in $[0,\infty)$ for every $n$. Moreover since $\varphi'(s)=\frac{1}{(1+s^2)^{3/2}}$ is decreasing, by the definition $\varphi_n(s)\ge \varphi_{n+1}(s)$ for every $n$ and for every $s\in [0,\infty)$. Property (b) is obvious by the definition of $\varphi_n$ and by the fact that $\Phi_n$ is a primitive of $\varphi_n$. Property (c) follows immediately by (a), indeed for every $n$ and $s$ 
$$
\Phi_n(s)=\int_0^s\varphi_n(\xi)d\xi\ge\int_0^s\varphi(\xi)d\xi = \Phi(s).
$$
To prove (d) we use that $\varphi_n$ is monotone increasing, hence for every $s\ge 0$
$$
\Phi_n(s) \le \varphi_n(s)\int_0^s d\xi = \varphi_n(s)s.
$$
Now, by the definitions of $\varphi_n$ and $\Phi_n$, there exist three suitable constants $a_n,\,b_n,\,c_n$ with $a_n>0$ such that $\varphi_n(s)s=(a_ns+b_n)s$ and $\Phi_n(s) = a_n\frac{s^2}{2}+b_n s+c_n$ for every $|s|\ge n$. Thus for every $s\ge n$ 
$$
0\le \varphi_n(s)s-\Phi_n(s)=\frac{a_n}{2}s^2-c_n\to +\infty\qquad\mbox{as }s\to \infty.
$$
Finally, for property (e) we observe that since $\varphi'$ is decreasing and $\varphi'(0)=\varphi_n'(0)=1$,  $\varphi_n(s)\le s$ for every $s\ge 0$. Now, $\varphi_n$ is strictly increasing, and so also its inverse $\varphi_n^{-1}$ has the same monotonicity. Thus, for every $s\ge 0$
$$
s=\varphi_n^{-1}(\varphi_n(s))\le \varphi_n^{-1}(s).
$$ 
This concludes the proof.
\end{proof}

For every $n\in\mathbb N$, we consider the approximated problem
\begin{equation}\label{eq:Pn}
\begin{cases}
-(\varphi_n(u'))'=a(x)f(u)\quad&\mbox{in }(0,1),\\
u>0&\mbox{in }(0,1),\\
u'(0)=u'(1)=0,&
\end{cases}
\end{equation}
where we recall that $a\in C^1([0,1])$, $a>0$ in $[0,1]$, and on $f\in C^1([0,\infty))$ we require only the two assumptions $(f_\mathrm{eq})$, $(f_\mathrm{sgn})$ given in the Introduction.
We now prove the following multiplicity result for \eqref{eq:Pn}.

\begin{theorem}\label{thm:Pn}
Let $n,\,k\in\mathbb N$, and let $a\in C^1([0,1])$, $a>0$ in $[0,1]$, and assume that $f\in C^1([0,\infty))$ satisfies $(f_\mathrm{eq})$ and $(f_\mathrm{sgn})$. If $f'(u_0)>\lambda_{k+1}$, then problem \eqref{eq:Pn} admits at least $2k$ non-constant classical solutions $u_{n,1},\dots,u_{n,2k}$. Furthermore, 
\begin{enumerate}[(a)]
\item
for every $j=1,\dots,k$, $u_{n,j}(0)<u_0$ and $u_{n,j}-u_0$ has exactly $j$ zeros.
\item
for every $\ell=1,\dots,k$, $u_{n,2k+1-\ell}(0)>u_0$ and $u_{n,2k+1-\ell}-u_0$ has exactly $\ell$ zeros.
\end{enumerate}
\end{theorem}

The proof of the above theorem relies on a shooting technique.
As a first step, we introduce $\hat{f}$ as the continuous extension to zero of $f$ on $(-\infty,0)$ and we consider the problem
\begin{equation}\label{eq:Pn>0}
\begin{cases}
-(\varphi_n(u'))'=a(x)\hat f(u)\quad&\mbox{in }(0,1),\\
u'(0)=u'(1)=0.&
\end{cases}
\end{equation}
We also define the primitive of $\hat f$ vanishing at $s = u_0$, i.e.
$\hat F(s):=\displaystyle{\int_{u_0}^s}\hat f(\xi)d\xi.
$
Notice that, by $(f_\mathrm{sgn})$, 
\begin{equation}\label{eq:Fprop}
\hat F(s)\ge 0 \mbox{ for every }s\in \mathbb R \qquad\mbox{and}\qquad  \hat F(s)=0 \; \Leftrightarrow s=u_0,
\end{equation}
\begin{equation}\label{eq:Fprop2}
\hat F \mbox{ is strictly decreasing on } [0,u_0] \quad\mbox{and}\quad \hat F \mbox{ is strictly increasing on } [u_0,+\infty).
\end{equation}
As a useful tool to investigate the qualitative properties of the solutions, we introduce, for every $n\in\mathbb N$, the {\it energy} of a function $u$ satisfying the equation in \eqref{eq:Pn>0} as 
\begin{equation}\label{eq:energy-un}
E_n(x):=u'(x)\varphi_n(u'(x))-\Phi_n(u'(x))+a(x)\hat F(u(x))\quad\mbox{for every }x\in[0,1].
\end{equation}
Notice that, from \eqref{eq:Fprop} together with Proposition \ref{prop:phin-Phin}-(d), it holds that $E_n(x)\ge 0$ for every $x\in[0,1]$. 

We are now in a position to make our shooting procedure effective. For every $d\ge 0$, we consider the Cauchy problem 
\begin{equation}\label{eq:PdC}
\begin{cases}
-(\varphi_n(u'))'=a(x)\hat f(u)\quad&\mbox{in }(0,1),\\
u(0)=d,\quad u'(0)=0.&
\end{cases} 
\end{equation}

For \eqref{eq:PdC}, the following global existence, uniqueness, and continuous dependence result holds.

\begin{lemma}\label{lem:uniqueness_Cauchy}
For every $d\in [0,+\infty)$, there exists a unique global solution $u_d$ of \eqref{eq:PdC} in $[0,1]$; moreover, $u_d$ is of class $C^2([0,1])$. In addition, if $(d_j)\subset [0,\infty)$ is such that $d_j\to d\in [0,\infty)$ as $j\to\infty$, then 
\begin{equation}\label{eq:dip_cont}
u_{d_j}\to u_{d}\quad\mbox{and}\quad \varphi_n(u'_{d_j})\to \varphi_n(u'_{d}) \quad \mbox{as }j\to\infty \end{equation}
uniformly in $[0,1]$.
\end{lemma}
\begin{proof}
We can rewrite the equation in \eqref{eq:PdC} as the following equivalent first-order planar system
\begin{equation}\label{eq:sys-Pn}
\begin{cases}
u'(x)=\varphi_n^{-1}(v(x))&\\
v'(x)=-a(x)\hat f(u(x))&
\end{cases}
\end{equation}
and we consider the Cauchy problem with initial conditions 
\begin{equation}\label{eq:ic-general}
u(x^0)=u^0,\quad v(x^0)=v^0.
\end{equation}
For any $(x^0,(u^0,v^0))\in [0,1]\times\mathbb R^2$, the existence and uniqueness of a local solution $(u,v)$ of \eqref{eq:sys-Pn}-\eqref{eq:ic-general} is guaranteed by the Cauchy-Lipschitz Theorem.

Concerning the global existence for \eqref{eq:PdC}, suppose by contradiction that there exists $x^*\in (x^0,1]$ such that $(u,v)$ is not defined for $x\ge x^*$, then 
\begin{equation}\label{eq:uu'toinfty}
\lim_{x\to (x^*)^-}(|u(x)|+|v(x)|)=\infty.
\end{equation}
Now, if we consider the energy along the solution as defined in \eqref{eq:energy-un}, we get, by Proposition \ref{prop:phin-Phin}-(d),
$$
|E_n'(x)|=|a'(x)|\hat F(u(x))\le C a(x)\hat F(u(x))\le C E_n(x),
$$
where $C:=\max_{[0,1]}|a'(x)|/\min_{[0,1]}a(x)$.
Therefore, by Gronwall's Lemma, 
\begin{equation}\label{eq:Gronwall-cons}
E_n(x)\le 
e^C E_n(x^0)\quad\mbox{for every }x\in [x^0,x^*).
\end{equation}
As $a(x)\hat F(u_d(x))\ge 0$ by \eqref{eq:Fprop}, from   
\eqref{eq:Gronwall-cons} we infer that
\begin{equation}\label{eq:u'd-bdd}
u'(x)\varphi_n(u'(x))-\Phi_n(u'(x))\le e^C E_n(x_0)\quad\mbox{for every }x\in [x^0,x^*).
\end{equation}
By Proposition \ref{prop:phin-Phin}-(d), this implies that $|u'|$ is bounded in $[x^0,x^*)$ and hence, also $u$ and $v$. This contradicts \eqref{eq:uu'toinfty}.

In the same way, we exclude the case where there exists $x^*\in [0,x^0)$ such that $(u,v)$ is not defined for $x\le x^*$. This implies that the solution must be defined on [0,1].

Finally, the regularity of $u$ is a consequence of the continuity of $a$ and $\hat f$, while the proof of the continuous  dependence is standard, once we have uniqueness and global existence.
\end{proof}

Furthermore, for \eqref{eq:Pn>0} we prove the following maximum principle-type result, which guarantees that all non-constant solutions of \eqref{eq:Pn>0} are positive. 

\begin{lemma}\label{lem:max-princ}
If $u$ is a classical solution of \eqref{eq:Pn>0}, then either $u\equiv -C$ for some $C\geq 0$, or $u>0$ in $[0,1]$.
\end{lemma}

\begin{proof}
This can be easily deduced from the uniqueness of the Cauchy problem as if $\min u = u(x_0)=-C$ with $C\geq 0$, then $u'(x_0)=0$. This implies that $u$  and  $-C$ are solutions of the Cauchy problem
$$
\begin{cases}
-(\varphi_n(w'))'=a(x)\hat f(w)\quad&\mbox{in }(0,1),
\\
w(x_0)=-C,\quad w'(x_0)=0.&
\end{cases} 
$$
By uniqueness, we conclude that  $u\equiv -C$.
\end{proof}

By Lemmas \ref{lem:uniqueness_Cauchy} and \ref{lem:max-princ}, if for some $d\in [0,\infty)$ the solution  $u_d$ of \eqref{eq:PdC} is non-constant and satisfies $u'_d(1)=0$, then $u_d$ solves \eqref{eq:Pn}. 
Therefore, our goal is to look for such initial data $d$. To this aim, set $v(x):=\varphi_n(u'(x))$, it is convenient to introduce  the following system of clockwise polar coordinates around the point $(u,v)=(u_0,0)$ 
\begin{equation}\label{eq:polar}
\begin{cases}
u(x)-u_0=\rho(x)\cos(\theta(x))&\\
v(x)=-\rho(x)\sin(\theta(x)).&
\end{cases}
\end{equation}
We remark that, in view of the uniqueness proved in Lemma \ref{lem:uniqueness_Cauchy}, either $\rho\equiv 0$ or $\rho(x)\not=0$ for all $x\in [0,1]$. Moreover, for $d\in [0,\infty)$, if $u_d$ solves \eqref{eq:PdC}, the corresponding angular variable $\theta_d$ satisfies the following differential equation in $(0,1)$,
\begin{equation}\label{eq:theta'}
\begin{aligned}
\theta_d'(x)&=\frac{1}{\rho_d^2(x)} \left[\varphi_n(u'_d(x))u'_d(x) + a(x)\hat f(u_d(x))(u_d(x)-u_0) \right]
\\
&=\frac{1}{\rho_d^2(x)} \left[v_d(x)\varphi_n^{-1}(v_d(x)) + a(x)\hat f(u_d(x))(u_d(x)-u_0) \right]
\end{aligned}
\end{equation}
with initial conditions
\begin{equation}\label{eq:initial-cond-polar}
\begin{cases}
\theta_d(0)=\pi \quad\text{and}\quad \rho_d(0)=|d-u_0|,&\text{if } d\in [0, u_0),
\\
\theta_d(0)=0 \quad\text{and}\quad \rho_d(0)=|d-u_0|,&\text{if } d\in (u_0, +\infty).
\end{cases}
\end{equation}
We further observe that by \eqref{eq:theta'} and $(f_{\mathrm{sgn}})$, $\theta_d'(x)>0$ for every $x\in [0,1]$. 

In the following lemma we prove that the hypothesis $f'(u_0)>\lambda_{k+1}$ allows us to estimate from below the number of half turns that a solution $(u_d,v_d)$ of 
\begin{equation}\label{eq:ic}
\begin{cases}
u'(x)=\varphi_n^{-1}(v(x))&\\
v'(x)=-a(x)\hat f(u(x))&\\
u(0)=d,\quad v(0)=0&
\end{cases}
\end{equation}
 performs around $(u_0,0)$ if it is shot from a point $(d,0)$ sufficiently close to $(u_0,0)$.

\begin{lemma}\label{lem:rotation}
If $f'(u_0)>\lambda_{k+1}$ for some $k\in\mathbb N$, then there exists $\bar \delta>0$ such that 
$$
\theta_d(1)-\theta_d(0)>k\pi\quad\mbox{for every }d\in [u_0-\bar\delta,u_0+\bar\delta]\setminus\{u_0\}.
$$ 
\end{lemma}

\begin{proof} 
By the assumption on $f'(u_0)$, for every $\bar\lambda\in (\lambda_{k+1},f'(u_0))$, there exists $\delta>0$ such that
\begin{equation}\label{eq:fs-1>lambda}
\hat f(s)(s-u_0)\ge \bar\lambda (s-u_0)^2\quad\mbox{for every }s\mbox{ such that }|s-u_0|<\delta.
\end{equation}
On the other hand, by the continuous dependence with respect to $d$ of \eqref{eq:PdC}, in correspondence of $\delta$ there exists $\bar\delta>0$ such that if $d\in [u_0-\bar\delta,u_0+\bar\delta]\setminus\{u_0\}$, $|u_d(x)-u_0|<\delta$ for every $x\in [0,1]$. Therefore, by \eqref{eq:theta'}, \eqref{eq:fs-1>lambda}, and Proposition~\ref{prop:phin-Phin}-(e), we get for every $d\in [u_0-\bar\delta,u_0+\bar\delta]\setminus\{u_0\}$ and for every $x\in [0,1]$
$$
\theta_d'(x)\ge \frac{1}{\rho_d^2(x)}\left[v_d^2(x)+\bar\lambda a(x)(u_d(x)-u_0)^2\right]=\sin^2(\theta_d(x))+\bar\lambda a(x)\cos^2(\theta_d(x)).
$$
Hence, by \eqref{eq:theta'-eigen} with $\lambda=\bar\lambda$ and using the Comparison Theorem for ODEs, we get 
$$
\theta_d(1) - \theta_d(0)\geq \vartheta_{\bar\lambda}(1)- \vartheta_{\bar\lambda}(0).
$$
Being $\lambda_{k+1}<\bar\lambda$, by \eqref{eq:theta-lambda-increas} and \eqref{eq:rotation-eigen} we deduce
$$
\theta_d(1) - \theta_d(0)> \vartheta_{\lambda_{k+1}}(1)- \vartheta_{\lambda_{k+1}}(0)=k\pi,
$$
that concludes the proof.
\end{proof}

The next lemma ensures that solutions $(u_d,v_d)$ with $d$ large enough are too ``slow'' to make a half turn.

\begin{lemma}\label{lem:apriori-n}
There exists $d^* > u_0$ such that the solution of \eqref{eq:ic} with $d=d^*$ satisfies $\theta_d(1)-\theta_d(0)<\pi$.
\end{lemma}

\begin{proof}
Otherwise, let $u$ be a solution of \eqref{eq:ic} with $d=d^*>u_0$ and $\theta_d(1)-\theta_d(0)\ge \pi$. Then, there exists $x_0\in (0,1]$ with $u'(x_0)=0$ and, for all $x\in (0,x_0)$, $u'(x)<0$.
We observe that $u(x_0)< u_0$ as otherwise, for every $x\in(0,x_0)$, $u(x)>u_0$ and hence $u''(x) <0$, contradicting $u'(0)=u'(x_0)=0$. Moreover, we can prove that  $u(x_0)>0$ in the same way as in Lemma  \ref{lem:max-princ}.

We now repeat the argument of the proof of Lemma \ref{lem:uniqueness_Cauchy} so as to obtain 
$$
E_n(x)\le 
e^C E_n(x_0)\quad\mbox{for every }x\in [0,x_0],
$$
where $C=\max_{[0,1]}|a'(x)|/\min_{[0,1]}a(x)$.
From this, recalling \eqref{eq:Fprop}-\eqref{eq:Fprop2} we get
$$
u'(x)\varphi_n(u'(x))-\Phi_n(u'(x)) < e^C a(x_0) F(u(x_0)) < e^C a(x_0) F(0) \quad\mbox{for all }x\in [0,x_0],
$$  
and hence, due to Proposition \ref{prop:phin-Phin}-(d),
$$
\vert u'(x) \vert \leq K^* \quad\mbox{for all }x\in [0,x_0],
$$
for a suitable $K^* > 0$ which does not depend on $u$. Then $u(x_0) \geq u(0) - K^* = d^* - K^*$, which is a contradiction if $d^* > u_0 + K^*$ (since $u(x_0) < u_0$).
\end{proof}

We are now ready to prove the main result of this section.
 
\begin{proof}[{\bf  Proof of Theorem \ref{thm:Pn}}]
Fix $n\in \mathbb N$. We first observe that the continuous dependence for \eqref{eq:ic} yields the continuity of the map $d\mapsto\theta_d(1)$, where $\theta_d$ is the angular variable of the solution $(u_d,v_d)$.
\medbreak

\noindent{\it First part: Existence of $u_{n,1}, \dots, u_{n,k}$.}
\medbreak

If $d=0$, the solution $(u_0,v_0)$ is identically equal to $(0,0)$, hence $\theta_0(1)-\theta_0(0)=0$. 

If $d\in [u_0-\bar\delta,u_0)$, by Lemma \ref{lem:rotation}, $\theta_d(1)-\theta_0(0)>k\pi$. Therefore, by continuity, there exist $k$ values of $d$, denoted by $d_{n,1},\, \dots,\,d_{n,k}$ such that 
$$
0<d_{n,1}<\dots<d_{n,k}<u_0\quad\mbox{and}\quad \theta_{d_{n,j}}(1)-\theta_{d_{n,j}}(0)=j\pi\;\mbox{ for every }j=1,\dots,k.
$$
As a result, to such $d_{n,j}$'s correspond $k$ solutions $u_{n,1}, \dots, u_{n,k}$ of the problem \eqref{eq:Pn>0}. Moreover, since $\theta_{d_{n,j}}$ is monotone increasing in $(0,1)$, cf. \eqref{eq:theta'}, for every $j=1,\dots,k$, there exist exactly $j$ points $0<x_{n,1}<\dots<x_{n,j}<1$ such that $\theta_{d_{n,j}}(x_{n,i})=\left(i+\frac{1}{2}\right)\pi$ for $i=1,\dots,j$. This means that $u_{n,j}(x_{n,i})=u_0$ for every $i=1,\dots,j$ and proves in particular that $u_{n,1}, \dots, u_{n,k}$ are $k$ distinct non-constant solutions of \eqref{eq:Pn>0}. Thus, by Lemma \ref{lem:max-princ}, they solve \eqref{eq:Pn} and have the desired oscillatory behavior. 
\medbreak

\noindent{\it Second part: Existence of $u_{n,k+1}, \dots, u_{n,2k}$.}
\medbreak

The argument is exactly the same with $d$ between $u_0$ and $d^*$ using Lemmas  \ref{lem:rotation} and  \ref{lem:apriori-n}.
\end{proof}

For further convenience, we prove here below an improved version of Lemma~\ref{lem:apriori-n}: precisely, we prove that we can take $d^* = \bar u$, when $f$ satisfies the additional assumption $(f_{\mathrm{ap}})$.

\begin{lemma}\label{lem:rotation u1}
Let $f$ satisfy also $(f_{\mathrm{ap}})$, then the solution of \eqref{eq:ic} with $d=\bar u$ satisfies $\theta_d(1)-\theta_d(0)<\pi$.
\end{lemma}

\begin{proof} 
Otherwise, there exists $x_0\in (0,1]$ with $u'(x_0)=0$ and, for all $x\in (0,x_0)$, $u'(x)<0$. As already observed in the proof of Lemma~\ref{lem:apriori-n}, it must be $0 < u(x_0)< u_0$.
\smallbreak

Now, for every $x\in[0,x_0]$, we get 
$$
\begin{aligned}
E_n(x)-E_n(x_0)&=\int_{x_0}^x a'(s)F(u(s))ds=\int_x^{x_0} \frac{a'^-(s)-a'^+(s)}{a(s)}a(s)F(u(s))ds\\
&\le \int_x^{x_0} \frac{a'^-(s)}{a(s)}E_n(s)ds.
\end{aligned}
$$
Therefore, by backward Gronwall's inequality, we obtain for all $x\in[0,x_0]$
$$
E_n(x)\le  E_n(x_0)
\exp \big(\int_x^{x_0}\frac{{a'}^-(s)}{a(s)}ds \big) .
$$
In particular, recalling \eqref{eq:Fprop}-\eqref{eq:Fprop2} we obtain
$$
\begin{aligned}
a(0) F(\bar u)=E_n(0)
&\le 
\exp \big(\int_0^{x_0} \frac{{a'}^-(x)}{a(x)}dx \big) 
a(x_0) F(u(x_0))
\\
&<
\exp \big(\int_0^{x_0} \frac{{a'}^-(x)}{a(x)}dx \big)  a(x_0) F(0).
\end{aligned}
$$
As 
$$
\max_{x_0\in[0,1] }\frac{a(x_0)}{a(0)}
\exp \big(\int_0^{x_0} \frac{{a'}^-(x)}{a(x)}dx \big)  =
\frac{a(1)}{a(0)}
\exp \big(\int_0^{1} \frac{{a'}^-(x)}{a(x)}dx \big),
$$
this  contradicts the choice of $\bar u$.
\end{proof}
 
\begin{remark}\label{rem:apriori}
In view of Lemma \ref{lem:rotation u1}, we can ensure that, when $f$ satisfies $(f_{\mathrm{ap}})$, all the solutions given by Theorem \ref{thm:Pn}
satisfy $u(0) < \bar u$. This is easily understood by checking the final part of the proof of Theorem \ref{thm:Pn}.  
\end{remark}

In the rest of the section, we fix $j \in \{1,\ldots,2k\}$ and we consider a sequence of solutions 
$(u_{n,j})_n$ given by Theorem \ref{thm:Pn}. We are going to establish, for such a sequence, an auxiliary property, which will play an important role in the next section. In what follows, since $j$ is fixed, to simplify the notation we simply write $u_n$ instead of $u_{n,j}$.

\begin{proposition}\label{prop:u-notuni-1}
The sequence $(u_n)$ does not have subsequences which converge uniformly to the constant function $u \equiv u_0$.
\end{proposition}

\begin{proof} 
As in the proof of Lemma  \ref{lem:rotation}, there exists $\delta>0$ such that, if $u$ is a solution of 
\eqref{eq:Pn}  with $|u(x)-u_0|<\delta$ for every $x\in [0,1]$, then the corresponding angular variable satisfies $\theta(1)-\theta(0)>k\pi$.

As $\theta_n(1)-\theta_n(0)=j\pi\leq k\pi$ this implies that, for all $n\in \mathbb N$, $\displaystyle \max_{[0,1]}|u_n(x)-u_0|\geq \delta$. This proves the result.
\end{proof}

As a corollary of the above result, we can further obtain the following bound for the minimum and maximum of the solutions
$u_n$.

\begin{corollary}\label{lem:max-min-1}
There exists $\bar\varepsilon>0$ such that for every $n\in\mathbb N$, and for every extremum point $\bar x\in [0,1]$ of $u_n$,
$$
|u_n(\bar x)-u_0|>\bar\varepsilon.
$$
In particular, $u_n(\bar x)>u_0+\bar\varepsilon$ if $\bar x$ is a relative maximizer, and  $u_n(\bar x)<u_0-\bar\varepsilon$ if $\bar x$ is a relative minimizer for $u_n$. 
\end{corollary}

\begin{proof}
Suppose by contradiction that in correspondence of $\varepsilon=\frac{1}{m}$ there exists $n(m)\in\mathbb N$ such that $u_{n(m)}$ has an extremum point, denoted by $x_{m}$, such that 
$$
|u_{n(m)}(x_m)-u_0|\le \frac{1}{m}.
$$
By uniqueness of the solution of the Cauchy problem $u_{n(m)}(x_m)\not=u_0$ as otherwise,  $x_m$ being an extremum, $u'_{n(m)}(x_m)=0$ and we would have two distinct solutions of the Cauchy problem $u(x_m)=u_0$ and $u'(x_m)=0$.
As every $u_n$ has exactly $j+1$ extremum points (counting also $x=0$ and $x=1$), the set $\{n(m)\,:\,m\in\mathbb N\}$ is unbounded. Thus, passing if necessary to a subsequence, $\lim_{m\to\infty}n(m)=\infty$.
Moreover, since $(x_m)\subset[0,1]$, up to a subsequence, $(x_m)$ converges to some point $\bar x \in[0,1]$. 

We claim that the corresponding subsequence of $(u_{n(m)})_m$ converges uniformly to $u_0$: in view of Proposition \ref{prop:u-notuni-1}, this will conclude the proof. To prove the claim, we consider the energy along the solution $u_{n(m)}$ defined in \eqref{eq:energy-un} and we argue as in the proof of Lemma \ref{lem:uniqueness_Cauchy} to get by Gronwall's Lemma
\begin{equation}\label{eq:est-gronwall}
E_{n(m)}(x)\le e^C E_{n(m)}(x_m)\quad\mbox{for every }x\in [0,1],
\end{equation} 
with $C=\max_{[0,1]}|a'(x)|/\min_{[0,1]}a(x)$. On the other hand, letting $m\to\infty$, 
$$
|u_{n(m)}(x_{m})-u_0|\to 0\quad\mbox{and consequently}\quad  E_{n(m)}(x_m)\to 0,
$$
where we have used again that $u'_{n(m)}(x_m)=0$. This, together with \eqref{eq:est-gronwall}, gives 
$$
E_{n(m)}\to 0 \quad \mbox{uniformly in }[0,1]\quad\mbox{as }m\to\infty.
$$
Using Proposition \ref{prop:phin-Phin}-(d) and the fact that $a$ is positive and $F$ is non-negative,  
the last convergence yields 
$$
\lim_{m\to\infty}\|F(u_{n(m)})\|_{L^\infty(0,1)}= 0.
$$
By \eqref{eq:Fprop}, we get that $u_{n(m)}\to u_0$ uniformly in $[0,1]$, that proves the claim. 

The last part of the statement follows by the concavity/convexity of $u_n$. Indeed, we know that $u_n$ solves $-\varphi_n'(u')u''=a(x)f(u)$. Hence, being $\varphi_n$ increasing, $a$ positive and using $(f_{\mathrm{sgn}})$, the solution $u_n$ is concave (resp., convex) in intervals in which $u_n>u_0$ (resp., $u_n<u_0$). Therefore, since a concave (resp., convex) function cannot have a minimum (rep. maximum) unless it is constant, $u_n(\bar x)>u_0$ (resp., $u_n(\bar x)<u_0$) if $\bar x$ is a relative maximizer (resp., minimizer) and the proof is concluded.
\end{proof}

\section{Proof of Part (I) - (II) of the main theorem}\label{sec:4}

Let $f'(u_0)>\lambda_{k+1}$ for some $k\in\mathbb N$ and fix any $j\in\{1,\dots,2k\}$. 
Consider the sequence $(u_{n,j})_n$ of solutions of \eqref{eq:Pn} having $\ell$ intersections with $u_0$ (where $\ell=j$ or $2k+1-j$ according to $j\leq k$ or $j>k$) whose existence has been proved in Theorem~\ref{thm:Pn}. As in the final part of the previous section, we will denote this sequence simply by $(u_{n})$. 

We are going to show that the sequence $(u_n)$ converges, in a suitable sense, to a BV-solution of \eqref{eq:Pmain}, having exactly
$\ell$ intersections (possibly in the generalized sense explained after the statement of Theorem \ref{thm:main}) with the constant $u_0$. 
This will be the core of the proof of Part (I) - (II) of Theorem \ref{thm:main} (the other statements in Part (I) - (II) follow as direct consequences of Proposition \ref{prop:Bvsol-description}). We split the next arguments into some steps.

\subsection{A priori estimates}

In the next two lemmas, we prove that the sequence $(u_n)$ is bounded in $W^{1,1}(0,1)$ when either $(f_{\mathrm{ap}})$ or
$(f_{\mathrm{ap}})'$ are satisfied.

\begin{lemma}\label{lem:bdd-Linfty}
Under condition $(f_{\mathrm{ap}})$, the sequence $(u_n)$ is bounded in $L^\infty(0,1)$: more precisely, for all $x\in [0,1]$, $0< u_n(x)< \bar u$. Moreover, $(u'_n)$ is bounded in $L^1(0,1)$ and, consequently, $(u_n)$ is bounded in $W^{1,1}(0,1)$. 
\end{lemma}

\begin{proof} For the first part, observe that, for every $n$, there exists $\overline x_n \in [0,1]$ such that $u_n(\overline x_n)\in (0, u_0)$, $u_n'(\overline x_n)=0$ and $u_n(\overline x_n)<u_n(x)<u_n(0)$ for all $x\in (0,\overline x_n)$.
Moreover, by Remark \ref{rem:apriori} we also have $u_n(0) < \bar u$ for every $n$.
We now denote by $E_n$ the energy of $u_n$, i.e. \eqref{eq:energy-un} with $u=u_n$. For every $x\in[0,1]$ we obtain
$$
E'_n(x)\le \frac{a'^+(x)-a'^-(x)}{a(x)}E_n(x)\le \frac{a'^+(x)}{a(x)}E_n(x). 
$$
Thus, by Gronwall's Lemma and using \eqref{eq:Fprop}-\eqref{eq:Fprop2}, 
we get for every $x\in [\overline x_n,1]$
$$
\begin{aligned}
a(x)F(u_n(x))\le E_n(x)
&\le  E_n(\overline x_n) \exp \big(\int_{\overline x_n}^x \frac{{a'}^+(s)}{a(s)}ds\big)
\\
&\leq   
a(\overline x_n)\exp \big(\int_{\overline x_n}^x \frac{{a'}^+(s)}{a(s)}ds\big)  F(0).
\end{aligned}
$$
As 
$$
\max_{\overline x_n, x\in[0,1] }\frac{a(\overline x_n)}{a(x)}
\exp \big(\int_{\overline x_n}^{x}\frac{{a'}^+(s)}{a(s)}ds \big)  =
\frac{a(0)}{a(1)}
\exp \big(\int_0^{1} \frac{{a'}^+(x)}{a(x)}dx \big),
$$
the first part of the thesis follows by the definition of $\bar u$ in $(f_{\mathrm{ap}})$.

As for the boundedness of $(u'_n)$ in $L^1(0,1)$, we proceed as follows. 
By Proposition~\ref{prop:phin-Phin}-(c), we get 
\begin{equation}\label{eq:un'}
\int_0^1 |u_n'|dx\le \int_0^1\sqrt{1+u_n'^2}dx\le \int_0^1 (\Phi_n(u_n')+1) dx.
\end{equation}
Since $u_n$ solves \eqref{eq:Pn}, $u_n$ is a global minimizer of the functional
$$
I_n(v):=\int_0^1\Phi_n(v')dx-\int_0^1a(x)f(u_n)vdx \quad v\in H^1(0,1),
$$
by convexity of $\Phi_n$ and consequently of $I_n$.
Therefore, for every $v\in H^1(0,1)$, by \eqref{eq:un'} we obtain 
$$
\int_0^1 |u_n'|dx\le \int_0^1\Phi_n(v')dx-\int_0^1a(x)f(u_n)(v-u_n)dx +1. 
$$
In particular, for $v\equiv 0$, 
$$
\int_0^1 |u_n'|dx\le \int_0^1a(x)f(u_n)u_n dx +1\le M, 
$$
for a suitable constant $M$, being $(u_n)$ bounded in $L^\infty(0,1)$ and $f$ continuous.

In conclusion, $(u_n)$ is bounded in $W^{1,1}(0,1)$ and so also in $BV(0,1)$.
\end{proof}

\begin{lemma}\label{lem:uprime-bdd-L1'}
Under condition $(f_{\mathrm{ap}})'$, the sequence $(u_n)$ is bounded in $W^{1,1}(0,1)$.
\end{lemma}

\begin{proof}
Consider as test function in \eqref{eq:Pn}, the function $v_n=u_n-\max u_n$, we obtain by Proposition \ref{prop:phin-Phin}-(a) 
$$
\begin{array}{rcl}
\displaystyle
\int_0^1 \varphi(u_n')u_n' \, dx
&\leq&
\displaystyle 
\int_0^1 \varphi_n(u_n')u_n' \, dx
=
\int_0^1 \varphi_n(u_n')v_n'\, dx
\\
&=& 
\displaystyle 
\int_0^1 a(x) f(u_n(x)) (u_n-\max u_n)\, dx
\\
&\leq&
\displaystyle 
\int_0^1 a(x) (\min f) (u_n-\max u_n)\, dx
\\
&\leq&
\displaystyle  
|\min u_n-\max u_n|  \|a\|_{L^1(0,1)} \max f^-
\\
&\leq&
\displaystyle  
\int_0^1 |u_n'|\,dx  \,\|a\|_{L^1(0,1)} \max f^-.
\end{array}
$$
Let $c$ such that, for all $s\in \mathbb R$, $\varphi(s)s-|s|\geq c$. Hence we deduce that
$$
\int_0^1 |u_n'|\,dx -c\leq  \|a\|_{L^1(0,1)} \max f^- \int_0^1 |u_n'|\,dx .
$$
As $\|a\|_{L^1(0,1)} \max f^-<1$, this proves that $\|u_n'\|_{L^1(0,1)}$ is bounded.
\medbreak

The result follows then from the fact that $0<\min u_n<u_0$ for every $n$, since  
$$
\|u_n\|_{L^\infty(0,1)}\le \min u_n+\sup_n\|u_n\|_{L^1(0,1)} < u_0+\sup_n\|u_n\|_{L^1(0,1)} <\infty.
$$
\end{proof}

\subsection{Passing to the limit}
By Lemmas \ref{lem:bdd-Linfty} and \ref{lem:uprime-bdd-L1'}, under either of the assumptions $(f_\mathrm{ap})$ and $(f_\mathrm{ap})'$, $(u_n)$ is bounded in $W^{1,1}(0,1)$.
It then follows (see, for instance, \cite[Theorem 3.23]{AFP}) that, up to subsequences, 
the sequence $(u_n)$ weakly-$\ast$ converges to $u$ in $BV(0,1)$, i.e., 
\begin{equation}\label{eq:lim}
u_n\to u\mbox{ in }L^1(0,1)\quad\mbox{and}\quad Du_n\overset{*}{\rightharpoonup}Du,
\end{equation}
where the last convergence means that $\displaystyle{\int_0^1\phi u_n'dx\to \int_0^1\phi dDu}$ for every $\phi\in C_0(0,1)$.

We now prove that the limit function $u$ is actually a solution of \eqref{eq:Pmain} in the BV-sense.

\begin{proposition}\label{prop:uBvsol}
The limit function $u$ is a BV-solution of \eqref{eq:Pmain}.
\end{proposition}
\begin{proof}
In view of Proposition \ref{prop:ineqW11ineqBV}, it suffices to show that the inequality 
\begin{equation}\label{eq:main-ineq}
\int_0^1\sqrt{1+v'^2}dx\ge\int_0^1\sqrt{1+|Du|^2}+\int_0^1 a(x)f(u)(v-u)dx
\end{equation}
holds for every $v\in C^\infty([0,1])$. We will prove that \eqref{eq:main-ineq} holds for every $v\in H^1(0,1)$. Indeed, arguing as in the proof of Lemma \ref{lem:bdd-Linfty}, since $u_n$ is a global minimizer of $I_n$, and using Proposition \ref{prop:phin-Phin}-(c), we get for every $v\in H^1(0,1)$
$$
\int_0^1\sqrt{1+u_n'^2}dx\leq \int_0^1 (\Phi_n(u_n')+1)dx\le \int_0^1 (\Phi_n(v')+1)dx-\int_0^1 a(x)f(u_n)(v-u_n)dx.
$$
Passing to the lower limit on both sides we get 
\begin{equation}
\label{eq:aux}
\liminf_{n\to\infty}\int_0^1\sqrt{1+u_n'^2}dx\le \int_0^1 \sqrt{1+v'^2}dx-\int_0^1 a(x)f(u)(v-u)dx,
\end{equation}
where we applied the Dominated Convergence Theorem to a subsequence of the right-hand side, recalling that $u_n\to u$ in $L^1(0,1)$, and we used that $\Phi_n\to\Phi$.
We know that the functional $J: BV(0,1)\to\mathbb R: v\mapsto \int_0^1\sqrt{1+|Dv|^2}$ is lower semicontinuous with respect to the $L^1$-convergence, cf. \cite[Thm 14.2]{Gi}, thus by \eqref{eq:lim},
$$
\int_0^1\sqrt{1+|Du|^2}\le \liminf_{n\to\infty}\int_0^1 \sqrt{1+u_n'^2}dx,
$$
which, together with \eqref{eq:aux} proves the thesis.  
\end{proof}

Actually, a stronger convergence of $u_n$ to $u$ can be established. More precisely:

\begin{proposition}\label{prop:strict-conv}
Up to a subsequence, the sequence $(u_n)$ converges strictly to $u$ in $BV(0,1)$, i.e., 
\begin{equation}\label{eq:strict-conv}
u_n\to u\mbox{ in }L^1(0,1)\quad\mbox{and}\quad \int_0^1 |u_n'|dx \to \int_0^1 |Du|,
\end{equation}
and furthermore, 
\begin{equation}\label{eq:conv-lungh-graph}
\int_0^1 \sqrt{1+u_n'^2}dx \to \int_0^1\sqrt{1+|Du|^2}.
\end{equation}
\end{proposition}
\begin{proof}
As $u_n\to u$ in $L^1(0,1)$,  if we prove \eqref{eq:conv-lungh-graph}, by \cite[Fact 3.1]{A} with $f(x,p)=|p|$, we immediately get also \eqref{eq:strict-conv}. By the lower semicontinuity of the functional $J$ defined in the proof of Proposition \ref{prop:uBvsol}, we have
$$
\int_0^1\sqrt{1+|Du|^2}\le \liminf_{n\to\infty}\int_0^1 \sqrt{1+u_n'^2}dx.
$$
On the other hand, since $u_n$ is a global minimizer of $I_n$ in $H^1(0,1)$, as in the proof of  Proposition \ref{prop:uBvsol}, we obtain, for every $v\in H^1(0,1)$,
$$
\int_0^1 \sqrt{1+u_n'^2}dx \le \int_0^1 (\Phi_n(u_n')+1)dx \le \int_0^1 (\Phi_n(v')+1)dx-\int_0^1 a(x)f(u_n)(v-u_n)dx.
$$
Therefore, passing to the upper limit on both sides, we get by Dominated Convergence Theorem applied if necessary to a subsequence
\begin{equation}\label{eq:limsup}
\limsup_{n\to\infty}\int_0^1 \sqrt{1+u_n'^2}dx \le \int_0^1 \sqrt{1+v'^2}dx-\int_0^1 a(x)f(u)(v-u)dx
\end{equation}
for all $v\in H^1(0,1)$. Since $ C^\infty([0,1])$ is dense in  $W^{1,1}(0,1)$, inequality \eqref{eq:limsup} actually holds for every $v\in W^{1,1}(0,1)$. 
Moreover, by Lemma \ref{lem:W11toBV},   there exists a sequence $(v_k)\subset W^{1,1}(0,1)$ such that $v_k\to u$ in $L^1(0,1)$ and $\int_0^1\sqrt{1+v_k'^2}dx\to\int_0^1\sqrt{1+|Du|^2}$. 
Thus, applying \eqref{eq:limsup} to $v_k$ and taking the limit as $k\to\infty$ we get
$$
\limsup_{n\to\infty}\int_0^1 \sqrt{1+u_n'^2}dx \le  \int_0^1\sqrt{1+|Du|^2},
$$
and the proof is concluded.
\end{proof}

As a consequence of Proposition \ref{prop:strict-conv}, we can establish a first important fact about the solution $u$.
Precisely: 

\begin{proposition}\label{prop:u-neq-1}
The solution $u$ is not identically equal to $u_0$.
\end{proposition}

\begin{proof}
Assume by contradiction that $u \equiv u_0$.
Then, by Proposition \ref{prop:strict-conv} it holds that $u_n \to u_0$ strictly in $BV(0,1)$, so that, in particular, 
$$
\lim_{n\to\infty} \int_0^1|u_n'|dx  = \int_0^1|Du_0| = 0.
$$
Now, with obvious meaning of $x_M$ and $x_m$,
$$
\max_{x\in[0,1]}u_n-\min_{x\in[0,1]}u_n=u_n(x_M)-u_n(x_m)=\int_{x_m}^{x_M}u'_n(x)dx\le \int_0^1|u_n'|dx,
$$
hence, we can infer that $u_n \to u_0$ uniformly on $[0,1]$. Since
this is excluded by Proposition \ref{prop:u-notuni-1}, the proof is concluded.
\end{proof}

The next step of the proof will consist of course in showing that the number of generalized intersections with $u_0$ is preserved when passing to the limit. This will be done in the next subsection, using in an essential way the following two results, which basically ensure that the convergence of $u_n$ to $u$ is stronger in the subintervals of $(0,1)$ where $(u'_n)$ is bounded.

\begin{proposition}\label{prop:vn-conv-unif}
Let $v_n=\varphi_n(u_n')$. Then, up to a subsequence, $(v_n)$ converges uniformly in $[0,1]$. As a consequence, denoted by $v$ the limit function, $v\in C([0,1])$. 
\end{proposition}
\begin{proof}
We first observe that, by the regularity of $u_n$ and of $\varphi_n$, $v_n=\varphi_n(u_n')\in C^1([0,1])$ for every $n$. If we prove that $(v_n')$ is bounded in $C^1([0,1])$, by the Arzel\`a-Ascoli Theorem, we get the thesis. 
By Lemmas \ref{lem:bdd-Linfty} and \ref{lem:uprime-bdd-L1'}, we know that $(u_n)$ is bounded in $L^\infty(0,1)$. By the second equation in \eqref{eq:sys-Pn}, we get for every $n\in\mathbb N$
$$
|v'_n(x)|\le \|a\|_{L^\infty(0,1)}\max_{s\in[0,\sup_n \|u_n\|_{L^\infty}]}f(s)=:C_2\quad\mbox{for every }x\in [0,1].
$$
Consequently, as $v_n(0)=0$, this gives
$\|v_n\|_{C^1([0,1])}\le 2C_2$ for every $n\in\mathbb N$ and concludes the proof.
\end{proof}

\begin{proposition}\label{prop:un-conv-unif}
Let $v$ be the limit function of $(v_n)$ introduced in Proposition~\ref{prop:vn-conv-unif} and let $[\alpha,\beta]\subseteq [0,1]$. If $|v|\le 1-\varepsilon$ in $[\alpha,\beta]$ for some $\varepsilon>0$, then up to a subsequence, $(u_n)$ converges uniformly to $u$ in $[\alpha,\beta]$ and consequently $u\in C([\alpha,\beta])$. 
\end{proposition}
\begin{proof}
Let $|v|\le 1-\varepsilon$ in $[\alpha,\beta]$, by the uniform convergence proved in Proposition~\ref{prop:vn-conv-unif}, for $n$ sufficiently large, 
$$
|v_n(x)|\le 1-\varepsilon'\quad\mbox{for every }x\in [\alpha,\beta],
$$
for some $\varepsilon'>0$. Consequently, for $n$ large, 
$$
\varphi_n^{-1}(v_n(x))=\varphi^{-1}(v_n(x))\quad\mbox{for all }x\in [\alpha,\beta].
$$
Hence, for $n$ large enough,
\begin{equation}
\label{eq:bound-u'n}
|u'_n(x)|=|\varphi^{-1}(v_n(x))|\le \varphi^{-1}(1-\varepsilon')\quad\mbox{for every }x\in [\alpha,\beta]
\end{equation}
Since by Lemmas \ref{lem:bdd-Linfty} and \ref{lem:uprime-bdd-L1'} we already know that $(u_n)$ is bounded in $L^\infty(0,1)$, \eqref{eq:bound-u'n} ensures that $(u_n)$ is bounded in $C^1([\alpha,\beta])$ and so, it admits a subsequence that converges uniformly to $u$ in $[\alpha,\beta]$. 
\end{proof}

\subsection{Proof of (I) and (II) of Theorem \ref{thm:main}}

For  every $n$, let $x_{n,1},\dots,x_{n,\ell}$ be the $\ell$ intersection points of $u_n$ with $u_0$, so that $0<x_{n,1}<\dots<x_{n,\ell}<1$. 
\medbreak

\noindent{\it Step 1: There exists $\delta>0$ such that, for all $n\in \mathbb N$,
\begin{itemize}
\item[(1.a)] $x_{n,1} >\delta$;
\item[(1.b)] $x_{n,i+1}-x_{n,i}>\delta$ for every $i=1,\dots,\ell-1$;
\item[(1.c)] $x_{n,\ell}<1-\delta$. 
\end{itemize}} 

By the uniform convergence of $(v_n)$ to $v$ (Proposition \ref{prop:vn-conv-unif}), as $u'_n(0)=v_n(0)=0$ for every $n$,  we have $v(0)=0$ and   there exist $\varepsilon$, $\delta>0$ such that $|v|\le 1-\varepsilon$ in $[0,\delta]$. Thus, by Proposition \ref{prop:un-conv-unif}, $u_n$ converges uniformly to $u$ in $[0,\delta]$. Now, by Corollary~\ref{lem:max-min-1}, $|u_n(0)- u_0|\geq  \bar\varepsilon$ for every $n$. Then also $|u(0) - u_0|\geq \bar\varepsilon$. Since $u$ is continuous in $[0,\delta]$, there exist $\varepsilon'$, $\delta'>0$ such that 
$|u- u_0| \geq \varepsilon'$  in $[0,\delta']$. This proves  (1.a) by the uniform convergence of $u_n$ to $u$ in $[0,\delta]$.
\medbreak

One  can argue similarly to prove (1.c). 
\medbreak

It remains to prove (1.b). For every $n$, $u_n(x_{n,i})= u_n(x_{n,i+1})=u_0$, hence, there exists $\bar x_n\in (x_{n,i}, x_{n,i+1})$ such that $u_n'(\bar x_n)=v_n(\bar x_n)=0$. Up to a subsequence 
$\lim_{n\to\infty}\bar x_{n}=:\bar x\in (0,1)$ and by
 Proposition \ref{prop:vn-conv-unif}, we get $v(\bar x)=0$. By continuity of $v$, there exist $\varepsilon$, $\delta>0$ such that $|v|\le 1-\varepsilon$ in $(\bar x-\bar\delta,\bar x+\bar\delta)$. Therefore, the convergence $u_n\to u$ is uniform in $(\bar x-\bar\delta,\bar x+\bar\delta)$, by Proposition \ref{prop:un-conv-unif}, and so, in particular,  $u_n(\bar x_n)\to u(\bar x)$ as $n\to \infty$.
Since $\bar x_n$ is an extremum point for $u_n$, by Corollary \ref{lem:max-min-1}, $|u_n(\bar x_n)-u_0|>\bar\varepsilon$ for every $n$, hence 
\begin{equation}\label{eq:u-1>eps}
|u(\bar x)-u_0|\ge \bar \varepsilon>0.
\end{equation}
By the uniform convergence of $(u_n)$ to $u$ on $(\bar x-\bar\delta,\bar x+\bar\delta)$, we deduce the existence of $\delta\leq \bar\delta$ such that, for all $x\in [\bar x-\delta,\bar x+\delta]$, $|u_n(x)-u_0|\geq \bar \varepsilon/2$ which proves (1.b).
\smallbreak

This concludes the proof of this first step.
\medbreak

\noindent{\it Step 2: For every $i=1,\ldots,\ell$, the sequence $(x_{n,i})$ has a limit denoted $x_i$}.
Let us prove it by recurrence on $i\in\{0,1,\ldots, \ell\}$ denoting $x_{n,0}=0$ and $x_{n,\ell+1}=1$.

Assume by contradiction the existence of $i\in \{1,\ldots, \ell\}$ such that 
$$
\lim_{n\to \infty} x_{n,i-1}=x_{i-1} \quad\mbox{ and } \quad
\liminf_{n\to \infty} x_{n,i}=:x_{i,1} < \limsup_{n\to \infty} x_{n,i}=:x_{i,2}.
$$
Let $(x_{n_k,i})$ be a subsequence converging to  $x_{i,1}$ and $(x_{n_p,i})$ be a subsequence converging to  $x_{i,2}$ and assume without loss of generality that
$$
\begin{cases}
u_{n}(x)>u_0,&\text{if } x\in (x_{n,i-1}, x_{n,i}),
\\
u_{n}(x)<u_0,&\text{if } x\in (x_{n,i}, x_{n,i+1}).
\end{cases}
$$
Observe that, by Step 1, $x_{i+1,1}=\liminf_{n\to \infty} x_{n,i+1}\geq x_{i,1}+\delta$ and $x_{i,1}\geq x_{i-1}+\delta$.

Let $[a,b]\subset (x_{i,1}, \min(x_{i,2},x_{i+1,1}))$. For $n$ large enough, we have also $[a,b]\subset (x_{n_k,i}, \min(x_{n_p,i},x_{n_k,i+1}))\subset  (x_{n_p,i-1}, x_{n_p,i})$ and hence, for all $x\in [a,b]$,
$$
u_{n_p}(x)>u_0>u_{n_k}(x).
$$
By the $L^1$ convergence of $(u_n)$ to $u$ we deduce that $u(x)=u_0$ for a.e. $x\in [a,b]$. 

Moreover, arguing in the same way, we easily prove that 
$u(x)\geq u_0$ a.e. on $(x_{i-1}, \min(x_{i+1,1}, x_{i,2}))$. By Proposition \ref{prop:Bvsol-description}, we then have that $u$ is concave on $(x_{i-1}, \min(x_{i+1,1}, x_{i,2}))$ and arguing as in Corollary \ref{cor:inter}, as $u(x)=u_0$ for a.e. $x\in [a,b]\subset (x_{i-1}, \min(x_{i+1,1}, x_{i,2}))$, we deduce that  $u=u_0$ on $(x_{i-1}, \min(x_{i+1,1}, x_{i,2}))$ which contradicts the existence of $\bar x\in [x_{i-1},   x_{i,2})$ such that $|u(\bar x)-u_0|\geq \bar \epsilon>0$.
\medbreak

\noindent{\it Step 3:  We have, for $ i=0,\dots,\ell $
\begin{equation}\label{eq:signs}
\begin{cases}
(-1)^{i+1}(u(x)-u_0) \geq 0 \text{ for every }x\in (x_{i},x_{i+1})  \text{ if }j\in \{1, \ldots, k\},
\\
(-1)^{i+1}(u(x)-u_0) \leq 0 \text{ for every }x\in (x_{i},x_{i+1}) \text{ if }j\in \{k+1, \ldots, 2k\}.
\end{cases}
\end{equation}
}
The argument is the same as in Step 2.
\medbreak

\noindent{\it Step 4:  Conclusion.}
We are now in the assumptions of Corollary \ref{cor:inter}.
Recalling that, by Proposition~\ref{prop:u-neq-1}, $u$ is not identically equal to $u_0$, we further infer that we have strict inequalities in \eqref{eq:signs}.
We have thus proved part b. of (I) (resp., (II)) of Theorem \ref{thm:main}. At this point, parts a. and c. directly follow from Proposition~\ref{prop:Bvsol-description}.

\section{Continuity of the energy - End of the proof of Theorem \ref{thm:main}}\label{sec:5}

As in the previous section, let $f'(u_0)>\lambda_{k+1}$ for some $k\in\mathbb N$ and fix any $j\in\{1,\dots,2k\}$. Consider the sequence $(u_{n,j})_n$ of solutions of \eqref{eq:Pn} given by Theorem~\ref{thm:Pn}.
For simplicity, we will denote this sequence by $(u_{n})$. 
Furthermore, we denote by $u$ the limit function of $(u_n)$ and by $x_1,\dots,x_{\ell}$ the (generalized or not) intersection points of $u$ with $u_0$. 

We recall that the energy of \eqref{eq:Pn} along $u_n$ is given by 
$$
E_n(x)=u_n'(x)\varphi_n(u_n'(x))-\Phi_n(u_n'(x))+a(x) F(u_n(x))\quad\mbox{for every }x\in[0,1]
$$
and that the  energy of \eqref{eq:Pmain} along $u$ is given by 
\begin{equation}\label{eq:energyPmain}
\mathcal E(x):=1-\frac{1}{\sqrt{1+(u'(x))^2}}+a(x)F(u(x)) \quad\mbox{for every }x\in D,
\end{equation}
where 
$D:=\{x\in[0,1]\,:\, u \text{ is continuous in } x\}$ and $F(u)=\int_{u_0}^u f(s) ds$. At the points of $D$ where $u$ has a vertical tangent, this definition has to be intended in the limit sense, i.e., $\mathcal E(x)=1+a(x)F(u(x))$.

The aim of this section is to prove that $\mathcal E$ can be continuously extended to the whole interval $[0,1]$. This will conclude the proof of Theorem \ref{thm:main}. We start with some preliminary results.

\begin{lemma}\label{lem:E_n-unif-to-E}
Up to a subsequence, $(E_n)$ converges uniformly in $[0,1]$. Consequently, denoted by $E$ the limit function,  $E$ is continuous in $[0,1]$. 
\end{lemma}
\begin{proof}
Arguing as in the proof of Lemma \ref{lem:uniqueness_Cauchy}, we get for every $n\in \mathbb N$
$$
|E_n'(x)|\le C E_n(x)\quad\mbox{for every }x\in [0,1],
$$ 
with $C=\max_{[0,1]}|a'(x)|/\min_{[0,1]}a(x)$, and so by Gronwall's Lemma
$$
E_n(x)\le e^CE_n(0)\le e^Ca(0)\max_{[0,\sup_n\|u_n\|_{L^\infty}]} F \quad\mbox{for every }x\in [0,1].
$$
Therefore, by the Arzel\`a-Ascoli Theorem, up to a subsequence $(E_n)$ converges uniformly in $[0,1]$ to some function $E$. The continuity of $E$ then follows by the continuity of $E_n$ for every $n$.
\end{proof}

\begin{lemma}\label{lem:v<1}
Let $v$ be the function defined in Proposition~\ref{prop:vn-conv-unif}. We have $|v(x)|\le 1$ for every $x\in [0,1]$ and, for $\bar x\in[0,1]\setminus\{x_1,\dots,x_{\ell}\}$,  $|v(\bar x)|<1$. 
\end{lemma}

\begin{proof}

\noindent{\it Step 1: $|v(x)|\le 1$ for every $x\in [0,1]$.}
\medbreak

\noindent
For instance, suppose by contradiction that $v(\bar x)>1$ for some $\bar x\in [0,1]$, then for $n$ large, $v_n(x)>1$ in $[\bar x-\delta, \bar x+\delta]$ for some $\delta>0$. This gives the contradiction  with  Lemmas \ref{lem:bdd-Linfty} and \ref{lem:uprime-bdd-L1'}
$$
\int_{\bar x-\delta}^{\bar x+\delta} u_n'(\xi)d\xi=\int_{\bar x-\delta}^{\bar x+\delta} \varphi_n^{-1}(v_n(\xi))d\xi\ge 2n\delta\to\infty\quad\mbox{as }n\to\infty.
$$
 
\noindent{\it Step 2: For $\bar x\in[0,1]\setminus\{x_1,\dots,x_{\ell}\}$,  $|v(\bar x)|<1$.}
\medbreak

\noindent
If $\bar x\in\{0,\,1\}$, $u'_n(\bar x)=0$ for every $n$, and the thesis is clearly verified. 
Let us consider the case $\bar x\in(0,1)\setminus\{x_1,\dots,x_{\ell}\}$. Suppose by contradiction that $|v(\bar x)|\ge 1$, then $(|u_n'(\bar x)|)$ is unbounded and so, up to a subsequence, $|u_n'(\bar x)|\to\infty$ as $n\to\infty$. Put $x_0:=0$ and $x_{\ell+1}:=1$, and let $i\in\{0,\dots,\ell\}$ be the integer such that $\bar x\in(x_i,x_{i+1})$. Fix $\varepsilon>0$ such that $\varepsilon<\min\{\bar x-x_i,x_{i+1}-\bar x\}$. Using the same notation as in the proof of Theorem \ref{thm:main}, since $x_{n,i}\to x_i$ and $x_{n,i+1}\to x_{i+1}$ as $n\to\infty$, for $n$ large,  
$\max\{|x_{n,i}-x_i|,|x_{n,i+1}-x_{i+1}|\}\le \varepsilon$. By the equation in \eqref{eq:Pn}, we know that for $n$ large, the functions $u_n$ are either all convex or all concave in the whole interval 
$[x_i+\varepsilon, x_{i+1}-\varepsilon]$.
Suppose, to fix the ideas, that  $u_n$ are  all convex  on $[x_i+\varepsilon, x_{i+1}-\varepsilon]$. Then,  for $n$ large, the following inequalities hold
$$
\begin{aligned}
u_n'(\bar x)&\le u_n'(x)\le u_n'(x_{i+1}-\varepsilon)\quad	&\mbox{if }x\in [\bar x, x_{i+1}-\varepsilon],\\
u_n'(\bar x)&\ge u_n'(x)\ge u_n'(x_i+\varepsilon)\quad	&\mbox{if }x\in [x_i+\varepsilon,\bar x].
\end{aligned}
$$
If $u_n'(\bar x)\to+\infty$, by Fatou's Lemma, 
$$
\begin{aligned}
\liminf_{n\to\infty}\int_{\bar x}^{x_{i+1}-\varepsilon}u_n'(x)dx&\ge \int_{\bar x}^{x_{i+1}-\varepsilon}\liminf_{n\to\infty} u_n'(x)dx\\
& \ge \int_{\bar x}^{x_{i+1}-\varepsilon}\liminf_{n\to\infty} u_n'(\bar x)dx=+\infty.
\end{aligned}
$$
This contradicts the fact that $(u_n')$ is bounded in $L^1(0,1)$. 
If $u_n'(\bar x)\to-\infty$, applying again Fatou's Lemma, we have
$$
\liminf_{n\to\infty}\int_{x_i+\varepsilon}^{\bar x}|u_n'(x)|dx\ge \int_{x_i+\varepsilon}^{\bar x}\liminf_{n\to\infty} |u_n'(x)|dx \ge \int_{x_i+\varepsilon}^{\bar x}\liminf_{n\to\infty} |u_n'(\bar x)|dx=+\infty.
$$
This yields again a contradiction and concludes the proof in this case. In case   $u_n$ are  all concave, the proof is analogous and we omit it.
\end{proof}

\begin{lemma}\label{lem:u'ntou'}
Let $[\alpha,\beta]\subset [0,1]\setminus\{x_1,\dots,x_{\ell}\}$, with $0\le \alpha<\beta\le 1$. Then, up to a subsequence, $(u_n)$ converges to $u$ in $C^1([\alpha,\beta])$.
\end{lemma}

\begin{proof}
By Lemma \ref{lem:v<1} and by the continuity of $v$ (see Proposition \ref{prop:vn-conv-unif}), we know that there exists $\varepsilon>0$ such that $|v(x)|\le 1-\varepsilon$ for every $x\in [\alpha,\beta]$. By the equation in \eqref{eq:Pn}, 
\begin{equation}\label{eq:est-for-u_n''}
|u_n''(x)|\le \frac{\|a\|_{L^\infty(0,1)}|f(u_n(x))|}{|\varphi_n'(u_n'(x))|}\quad\mbox{for every }x\in [0,1].  
\end{equation}
Proceeding as in the proof of Proposition \ref{prop:un-conv-unif}, we have for $n$ large 
\begin{equation}\label{eq:est-u_n'}
|u_n'(x)|\le \varphi^{-1}(1-\varepsilon')\quad\mbox{for every }x\in [\alpha,\beta]\mbox{ and for some } \varepsilon'>0.
\end{equation}
Therefore, for $n$ large $\varphi_n'(u_n')=\varphi'(u_n')$ in $[\alpha,\beta]$, and we obtain
$$
|\varphi'(u_n'(x))|\ge|\varphi'(\varphi^{-1}(1-\varepsilon')| >0\quad\mbox{for every }x\in[\alpha,\beta].
$$
On the other hand, by Lemmas \ref{lem:bdd-Linfty} and \ref{lem:uprime-bdd-L1'}, we know that $\sup_n\|u_n\|_{L^\infty(0,1)}<\infty$.
Hence, from \eqref{eq:est-for-u_n''} we get for $n$ large and for every $x\in [\alpha,\beta]$,
\begin{equation}\label{eq:est-u_n''}
|u_n''(x)|\le \frac{\|a\|_{L^\infty(0,1)}\max_{s\in[0,\sup_n\|u_n\|_{L^\infty}]}|f(s)|}{|\varphi'(\varphi^{-1}(1-\varepsilon'))|}. 
\end{equation}
Combining together \eqref{eq:est-u_n'} and \eqref{eq:est-u_n''}, we can apply the Arzel\`a-Ascoli theorem and conclude that, up to a subsequence, $(u_n')$ converges uniformly to some function $w$ in $[\alpha,\beta]$. On the other hand, by Proposition \ref{prop:un-conv-unif}, we know that $(u_n)$ converges uniformly to $u$ in $[\alpha,\beta]$. Thus $w=u'$ in $[\alpha,\beta]$.
\end{proof}

\begin{theorem}\label{thm:cont-en}
$E(x)=\mathcal E(x)$ for every $x\in [0,1]\setminus\{x_1,\dots,x_j\}$.
In particular, $\mathcal E$ can be extended by continuity as $E(x)$ at every point $x\in[0,1]\setminus D$.
\end{theorem}
\begin{proof}
Let $\bar x\in [0,1]\setminus\{x_1,\dots,x_{\ell}\}$. If $\bar x\in\{0,1\}$ the thesis is verified, since $u_n(0)\to u(0)$ and $u_n(1)\to u(1)$, cf. the proof of Theorem \ref{thm:main}- Step 1 (1.a) and (1.c). Otherwise, let $\delta>0$ be such that $[\bar x-\delta,\bar x+\delta]\subset [0,1]\setminus\{x_1,\dots,x_j\}$. By Lemma~\ref{lem:u'ntou'}, 
$$
a(\cdot)F(u_n)\to a(\cdot)F(u)\quad\mbox{uniformly in }[\bar x-\delta,\bar x+\delta]
$$ 
and
$$
u_n'\varphi_n(u_n')-\Phi_n(u_n')\to 1-\frac{1}{\sqrt{1-(u')^2}}\quad\mbox{uniformly in }[\bar x-\delta,\bar x+\delta].
$$ 
Therefore, by Lemma \ref{lem:E_n-unif-to-E} and the uniqueness of the limit,
$$
E(x)= 1-\frac{1}{\sqrt{1-(u'(x))^2}}+a(x)F(u(x))\quad\mbox{for every }x\in [\bar x-\delta,\bar x+\delta].
$$
\end{proof}

\section{Existence of classical solutions}\label{sec:6}

In this section, we give a result ensuring that the solutions of \eqref{eq:Pmain} found by the previous approximation procedure are actually classical solutions. The precise statement is the following.

\begin{theorem}\label{thm:classical}
Let $k\in \mathbb N$, $a\in C^1([0,1])$, $a>0$ in $[0,1]$ and assume that $f\in C^1([0,\infty))$ satisfy $(f_\mathrm{eq})$, $(f_\mathrm{sgn})$ and $f'(u_0)>\lambda_{k+1}$.

If moreover  
\begin{equation}
\label{eq:f-cl-n}
a(0)\exp(\int_0^1 \frac{a'^+(x)}{a(x)} dx) \int_{u_0}^0 f(s)ds <1
\end{equation}
then, there exist at least $k$ non-constant $C^2$-solutions $u_1$, ... , $u_k$ of \eqref{eq:Pmain}, having the properties stated in Theorem \ref{thm:main}-(I).

On the other hand, if  
\begin{equation}
\label{eq:f-cl-n'}
a(0)\exp(\int_0^1 \frac{a'^+(x)}{a(x)} dx) \int_{u_0}^{+\infty} f(s)ds <1
\end{equation}
then, there exist at least $k$ non-constant $C^2$-solutions $u_{k+1}$, ... , $u_{2k}$ of \eqref{eq:Pmain}, having the properties stated in Theorem \ref{thm:main}-(II).
\end{theorem}

\begin{proof} Let us consider the first case, the second one is similar.
\medbreak

We are going to show that, for $n$ sufficiently large,
it holds that
\begin{equation}\label{eq:uclass1}
\vert u_{n,i}'(x) \vert \leq n,\quad\mbox{for all }x\in [0,1].
\end{equation}
Since $\varphi_n(s) = \varphi(s)$ for $\vert s \vert \leq n$, this implies that, for $n$ large enough, the $C^2$-function $u_{n,i}$ is a solution of \eqref{eq:Pmain},
thus concluding the proof.

In order to prove \eqref{eq:uclass1}, we first argue similarly as in the proof of Lemma \ref{lem:uniqueness_Cauchy} to find the estimate
$$
E_n'(x) \leq  \frac{{a'}^+(x)}{a(x)} E_n(x), \quad\mbox{for all }x\in (0,1),
$$
which in turn implies that
$$
E_n(x) \leq E_n(0) \exp\big(\int_0^x  \frac{{a'}^+(s)}{a(s)}ds\big),\quad\mbox{for all }x\in (0,1).
$$
From this, setting $K_n(s) := s \varphi_n(s) - \Phi_n(s)$, we obtain
$$
\begin{aligned}
K_n(u_{n,i}'(x)) & \le E_n(0)  \exp\big(\int_0^x \frac{{a'}^+(s)}{a(s)} ds\big) - a(x) F(u_{n,i}(x)) \\
& =  a(0) F(d_{n,i})  \exp\big(\int_0^x \frac{{a'}^+(s)}{a(s)} ds\big) - a(x) F(u_{n,i}(x)), 
\end{aligned}
$$
where we have denoted $d_{n,i}=u_{n,i}(0)$.
Thus, recalling \eqref{eq:Fprop}, we infer that
$$
K_n(u_{n,i}'(x)) \leq a(0)  \exp\big(\int_0^x  \frac{{a'}^+(s)}{a(s)} ds\big) F(d_{n,i}),\quad\mbox{for all }x\in (0,1),
$$
and finally, by assumption \eqref{eq:f-cl-n} and by \eqref{eq:Fprop2},
\begin{equation}\label{eq:uclass2}
K_n(u_{n,i}'(x)) \leq 1-\eta,\quad\mbox{for all }x\in (0,1),
\end{equation}
where $\eta = 1 - a(0) \exp\big(\int_0^1  \frac{{a'}^+(x)}{a(x)}dx\big) F(0)>0$.

On the other hand, since the function $K_n$ is increasing for $s \geq 0$ and decreasing for $s \leq 0$, 
a simple computation yields
\begin{equation}\label{eq:uclass3}
K_n(s) \geq K_n(n) = 1 - \frac{1}{\sqrt{1+n^2}}, \quad\mbox{for }|s| \geq n.
\end{equation}
Combining \eqref{eq:uclass2} and \eqref{eq:uclass3}, the estimate \eqref{eq:uclass1} easily follows for
$n \geq \sqrt{1/\eta^2-1}$.
\end{proof}

\begin{remark}\label{rmk:cl_1-2}
Observe that, in case $(f_\mathrm{ap})$ holds, as we know that for $i\in \{k+1, \ldots, 2k\}$, $d_{n,i}\in [u_0, \bar u]$, the condition \eqref{eq:f-cl-n'}
can be replaced by
$$
a(0)\exp(\int_0^1 \frac{a'^+(x)}{a(x)} dx) \int_{u_0}^{\bar u} f(s)ds <1.
$$

In the same way, in case $(f_\mathrm{ap})'$ holds, let $R$ be given by Lemma \ref{lem:uprime-bdd-L1'} such that, for all $n\in \mathbb N$, $\|u_n\|_{L^\infty(0,1)}\leq R$. Then the condition
\eqref{eq:f-cl-n'} can be replaced by
$$a(0)\exp(\int_0^1 \frac{a'^+(x)}{a(x)} dx) \int_{u_0}^{R} f(s)ds <1.
$$
\end{remark}

\begin{remark}
Considering the proof of Theorem \ref{thm:classical}, it seems natural  that   the solutions of \eqref{eq:Pmain} having a large number of intersections with $u_0$ are classical while the solutions   having a low number of intersections are only BV-solutions.
\end{remark}

\begin{remark}
In the autonomous case, that is $a(x) \equiv a$, in case $(f_\mathrm{ap})$ holds, one has $\int_{u_0}^{\bar u}f(s)ds = \int_{u_0}^0f(s)ds$ and so, also in view of Remark \ref{rmk:cl_1-2}, conditions \eqref{eq:f-cl-n} and \eqref{eq:f-cl-n'} reduce to
\begin{equation}\label{cond:class}
a \int_{u_0}^0 f(s)ds < 1.
\end{equation}
The above condition has a clear dynamical interpretation. Indeed, it means that the planar system \eqref{ham-sys} admits a classical homoclinic orbit to the equilibrium point $(0,0)$ (incidentally, notice that $(\bar u,0)$ is nothing but the intersection point of the homoclinic with the positive $u$-semiaxis). Since, as already discussed in Figure \ref{fig}, all the Neumann solutions must lie inside the region bounded by this homoclinic orbit, it is immediately understood that they have to be classical solutions. 

We also notice that, for a fixed nonlinear term $f$, condition \eqref{cond:class} is always satisfied for $a$ sufficiently small and never satisfied when $a$ is sufficiently large. 
More explicit conditions can be given for particular choices of the function $f$. For instance, in the model example
$$
f(s)=-\lambda s +  s^p,\quad\mbox{with }p>1\mbox{ and }\lambda>0,
$$
it turns out that $u_{0}=\lambda^{\frac{1}{p-1}}$ and a simple computation shows that \eqref{cond:class} is satisfied if and only if
\begin{equation}\label{eq:suff-auton}
a \lambda^{\frac{p+1}{p-1}}\frac{p-1}{2(p+1)}<1.
\end{equation}
Assuming for simplicity $a = 1$, we thus see that, if $\lambda\le 1$, \eqref{eq:suff-auton} is automatically verified and all the solutions found in Theorem \ref{thm:main} are classical. On the contrary, if $\lambda>1$, \eqref{eq:suff-auton} is not automatic and it is in competition with the assumption required in Theorem \ref{thm:main} for the existence of at least one non-constant possibly discontinuous BV-solution of \eqref{eq:Pmain}, i.e., $f'(u_{0})>\lambda_2$, or equivalently 
\begin{equation}\label{eq:suff-exist}
(p-1)\lambda >\lambda_2.
\end{equation}
So, in this case, the intersection of the two assumptions \eqref{eq:suff-auton} and \eqref{eq:suff-exist} is given by 
\begin{equation}\label{eq:suff-interval}
\frac{\lambda_2}{p-1}<\lambda<\left(\frac{2(p+1)}{p-1}\right)^{\frac{p-1}{p+1}}
\end{equation}
which is certainly not empty for $p$ large. 

We finally observe that, if we set the problem \eqref{eq:Pmain} in the interval $(0,L)$, instead of $(0,1)$, and we let $L \to \infty$, the eigenvalues $\lambda_k \to 0$. Thus, condition \eqref{eq:suff-interval} is not empty also when $L$ is sufficiently large.
\end{remark}

\begin{remark}
Recall also that, by \cite[Corollary 3.5]{LGO}, if 
$$
\int_0^1 a(x)|f(u(x))|\,dx<1
$$
then the solution $u$ is classical.
\end{remark}

\section*{Acknowledgments}
The authors thank warmly D. Bonheure for helpful discussions.
This research was partially supported by the INdAM - GNAMPA Project 2019 ``Il modello di Born-Infeld per l'elettromagnetismo nonlineare: esistenza, regolarit\`a e molteplicit\`a di soluzioni" and the INdAM - GNAMPA Project 2020 ``Problemi ai limiti per l'equazione della curvatura media prescritta''. 
C. De Coster and F. Colasuonno acknowledge respectively the supports of the Department of Mathematics - University of Turin and of the LAMAV - Universit\'e Polytechnique Hauts-de-France for their visit at Turin and Valenciennes, where parts of this work have been achieved.

\bibliographystyle{abbrv}
\bibliography{biblio}

\end{document}